\documentclass[11pt]{amsart}

\usepackage{a4wide}
\usepackage{amsfonts}
\usepackage{psfrag}

\usepackage{amsmath}
\usepackage{amssymb}
\usepackage{amsthm}
\usepackage{graphicx}
\usepackage{color}

\newtheorem {definition}{Definition}
\newtheorem{proposition}[definition]{Proposition}
\newtheorem {lemma}[definition]{Lemma}

\newtheorem {thm}[definition]{Theorem}
\newtheorem {rem}[definition]{Remark}

\numberwithin{definition}{section}
\numberwithin{equation}{section}

\newcommand{\cA}{{\mathcal A}}
\newcommand{\cB}{{\mathcal B}}

\newcommand{\cD}{{\mathcal D}}

\newcommand{\cG}{{\mathcal G}}
\newcommand{\cH}{{\mathcal H}}
\newcommand{\cN}{{\mathcal N}}

\newcommand{\cP}{{\mathcal P}}
\newcommand{\cQ}{{\mathcal Q}}

\newcommand{\D}{\mathbb{D}}
\newcommand{\E}{E}
\newcommand{\N}{\mathbb{N}}
\renewcommand{\P}{\mathbb{P}}
\newcommand{\EE}{\mathbb{E}}
\newcommand{\R}{\mathbb{R}}
\newcommand{\eps}{\varepsilon}

\newcommand{\dd}{\mathsf{d}}
\newcommand{\ee}{\mathsf{e}}

\newcommand{\mvb}{\mathversion{bold}}

\begin{document}
\title{A Jump Type SDE Approach \\to Positive Self-Similar Markov Processes}
\author{Leif D\"oring}
\thanks{L. D\"oring has been supported by the Foundation Science Mat\'ematiques de Paris}
\address{Leif D\"oring, Laboratoire de Probabilit\'es et Mod\`{e}les Al\'eatoires Universit\'e Paris 6, 4 place Jussieu, 75252 Paris Cedex 05, France}
\email{leif.doering@googlemail.com}
\thanks{
 This work has been undertaken while M. Barczy was on a post-doctoral position
 at the Laboratoire de Probabilit\'es et Mod\`{e}les Al\'eatoires,
 University Pierre-et-Marie Curie, Paris, thanks to NKTH-OTKA-EU FP7 (Marie Curie action)
 co-funded 'MOBILITY' Grant No. OMFB-00610/2010.
 M. Barczy has been also supported by the Hungarian Scientific Research Fund under Grant No.\ OTKA T-079128.
}
\author{M\'aty\'as Barczy}
\address{M\'aty\'as Barczy, University of Debrecen, Faculty of Informatics, Pf.12, H-4010 Debrecen, Hungary}
\email{barczy.matyas@inf.unideb.hu}



\begin{abstract}
We present a new approach to positive self-similar Markov processes (pssMps) by
 reformulating Lamperti's transformation via jump type SDEs.
As applications, we give direct constructions of pssMps (re)started continuously at zero  if the Lamperti transformed L\'evy process is spectrally negative. Our paper can be seen as a continuation of similar studies for continuous state branching processes.
\end{abstract}

\maketitle

\section{Introduction}

Positive self-similar Markov processes (pssMps) play an important role in various branches of probability theory such as branching processes and fragmentation theory. In recent years two problems concerning the behavior of pssMps at zero have attracted a lot of attention.
Solutions were heavily based on L\'evy process theory.
The main objective of this paper is to derive a jump type stochastic differential equation (SDE) for
 pssMps and prove well-posedness.
As application of this approach we derive a new representation of pssMps (re)started continuously at zero
 by applying stochastic calculus.

\subsection{Lamperti's transformation and ''problems at zero''}

We will use the following canonical notations.
Let $\R_+$ denote the set of non-negative real numbers
 and let $\D$ be the space of c\`{a}dl\`{a}g functions $\omega:\R_+\to\R_+$ (right continuous with left limits)
 endowed with the Borel sigma-field $\cD$ generated by Skorokhod's $J_1$ topology.
Let us consider a family of probability measures $\{\P_z,z\geq 0\}$ on $(\D,\cD)$ under which the coordinate
 process $Z_t(\omega):=\omega(t)$, $t\geq 0$, is a strong Markov process starting from $z$ and fulfills
 the following scaling property: There exists a constant $a>0$, called the index of self-similarity, such that
 \begin{align}\label{self_sim}
   \text{the law of $(c^{-a}Z_{ct})_{t\geq 0}$ under $\P_z$ is $\P_{c^{-a}z}$}
 \end{align}
 for all $c>0$ and $z\geq 0$.
These laws are called non-negative self-similar Markov distributions and simultaneously the canonical process
 $Z$ on the probability space $(\D,\cD,\P_z)$ is called a non-negative self-similar Markov process started
 from $z$.
Let us denote the corresponding process absorbed at the (possibly infinite) first hitting time $T_0$ of
 zero by
 \begin{align*}
	Z^\dag_t:=Z_t \mathbf 1_{\{t \leq T_0\}},\quad t\geq 0, \quad\text{with}\quad T_0:=\inf\{t\geq 0: Z_t=0\}.
 \end{align*}
The laws of the process $(Z^\dag_t)_{t\geq 0}$ under $\{\P_z,z\geq 0\}$, denoted by $\{\P^\dag_z, z\geq 0\}$,
 are non-negative self-similar Markov distributions with the same index of self-similarity.
They will be called positive self-similar Markov distributions and simultaneously, $Z^\dag$ is called
 a positive self-similar Markov process (pssMp).
Note that $\P_0^\dag$ is the law of the process identical to $0$, and that
 $0$ is a trap for $Z^\dag$, i.e. the corresponding laws $\{\P^\dag_z, z\geq 0\}$ are
 carried by $\{\omega\in\D : \omega(t)=0, \forall\, t\geq T_0(\omega)\}$.
Alternatively, by a family of positive self-similar Markov distributions we mean a self-similar
 strong Markov family carried by $\{\omega\in\D : \omega(t)=0, \forall\, t\geq T_0(\omega)\}$.
It is important to note that, by our convention, the difference between positive and non-negative self-similar
 Markov processes is only that pssMps always have $0$ as a trap.
The formulation of pssMps might seem inconvenient but allows us for a more natural formulation of the results.
\medskip

The systematic study of pssMps goes back to the 1960s, in particular to Lamperti's seminal article
 \cite{L}.
Lamperti \cite[Lemmas 2.5 and 3.2]{L} showed that the following trichotomy holds:
	\begin{enumerate}
		\item either $\P^\dag_z(T_0=\infty)=1$ for all $z>0$,
		\item or $\P^\dag_z(T_0<\infty\text{ and }Z_{T_0-}=0)=1$ for all $z>0$,
		\item or $\P^\dag_z(T_0<\infty\text{ and }Z_{T_0-}>0)=1$ for all $z>0$.
	\end{enumerate}
Let us next recall what is typically referred to as Lamperti's transformation (or Lamperti's representation).
If $Z$ is a pssMp of self-similarity index $a>0$ starting from $z>0$, then there is a
 possibly killed L\'evy process $(\xi_t)_{t\geq 0}$ {(whose law does not depend on $z$)}
 such that the canonical process $Z$ can be written as
 \begin{align}\label{LT}
   Z_t= z\exp\big(\xi_{\tau({tz^{-1/a})}} \big),\qquad 0\leq t<T_0,
 \end{align}
 where 	
 \begin{align*}
	\tau(t):=\inf\{s\geq 0:I_s\geq t\}
  \qquad \text{and}\qquad I_t:=\int_0^t \exp\left(\frac{1}{a}\xi_s\right) \dd s,\qquad t\geq 0.
	\end{align*}
Since the transformation is invertible it yields a bijection between the class of pssMps
 and that of L\'evy processes and furthermore it implies
 \begin{align*}
  &(1) \qquad \Longleftrightarrow \qquad \text{$\xi$ has infinite lifetime and}\;\,
               \P^\dag_z(\limsup_{t\to\infty}\xi_t=\infty)=1, \forall\; z>0,\\
  &(2) \qquad \Longleftrightarrow \qquad \text{$\xi$ has infinite lifetime and}\;\,
               \P^\dag_z(\lim_{t\to\infty}\xi_t=-\infty)=1, \forall\; z>0,\\
 &(3)\qquad \Longleftrightarrow \qquad  \text{$\xi$ has finite lifetime, i.e.
                                      is killed {at} rate $q>0$ {with $-\infty$ as a cemetery.}}
 \end{align*}
Thus, with our definition of pssMps, (\ref{LT}) holds equally for $t\geq 0$ with the convention
 $\tau(t z^{-1/a})=\infty$ for $t\geq T_0$.

Lamperti's representation has two drawbacks: first, it has no counterpart for non-negative self-similar Markov processes that do not have $0$ as a trap since the infinite time-horizon $[0,\infty)$ is compressed via $\tau$
 to the possibly finite time-horizon $[0,T_0)$ and {for such a non-negative self-similar Markov process}
 \eqref{LT} is valid only up to $T_0$.
Hence, the entire path of $\xi$ is already used to describe $Z$ until $T_0$.
Secondly, in all cases Lamperti's representation holds for strictly positive initial conditions $z>0$ but
 not for $z=0$.
Natural questions are then how to describe self-similar extensions of $(Z_t)_{t\in[0,T_0)}$ after $T_0$
 and how to describe self-similar Markov processes started from zero.
In the sequel we shall refer to these two problems as the "problems at zero".
The problems at zero have been resolved in recent years based on criteria involving the
 Lamperti transformed L\'evy process $\xi$.

Apparently, the first problem only occurs if $\xi$ drifts to $-\infty$ {- which includes being killed in {an almost surely}
 finite time - } as otherwise $T_0=\infty$ almost surely.
In the special case when $Z$ is a Brownian motion killed at $0$, the problem was already solved by Lamperti
 \cite[Theorem 5.1]{L} and then the general case subsequently was tackled by Rivero \cite{R1} and
 Vuolle-Apiala \cite{VA}.
Finally, it has been completely solved by Fitzsimmons \cite[{Theorem 1}]{F} and Rivero \cite[{Theorem 2}]{R2}
 who have independently proved that {the existence of a unique recurrent self-similar extension that leaves
 $0$ continuously is equivalent to the following Cram\'er type condition
  \begin{align}\label{cramer}
		\text{there is a }0<\theta<\frac{1}{a}\text{ such that }\EE\big(\ee^{\theta \xi_1}; \zeta>1\big)=1,
 \end{align}
 where $\zeta$ denotes the possibly infinite killing time of $\xi$.}

Since the law of $\xi$ does not depend on $z>0$, in condition \eqref{cramer}
 we simply wrote $\EE$ instead of $\EE^\dag_z$ (omitting also $\dag$).
We recall that a recurrent self-similar extension of a pssMp $\{\P_z^\dag, z\geq 0\}$ is a non-negative
 self-similar Markov process $\{\P_z, z\geq 0\}$ on $(\D,\cD)$ of same index of self-similarity
 such that the law of the process $(Z^\dag_t)_{t\geq 0}$ under $\{\P_z, z\geq 0\}$ equals the law of
 $(Z_t)_{t\geq 0}$ under $\{\P_z^\dag, z\geq 0\}$ and $0$ is not a trap for $\{\P_z, z\geq 0\}$.
The recurrent self-similar extension $\{\P_z, z\geq 0\}$ is said to leave zero continuously
 if additionally
 \[
    \P_0(Z_s =0 \;\,\text{for all $s\in G$})=1,
 \]
 where $G$ denotes the set of strictly positive left endpoints of the maximal intervals in the complement
 of the closure of the zero set $\{t\geq 0 : Z_t =0\}$.
For more information, in particular on the uniqueness of the extensions, see Fitzsimmons \cite{F}
 and Rivero \cite{R1}, {\cite{R2}}.
In both articles, the recurrent extensions were characterized via their excursion measures.

The second problem can be formalized as follows: Does weak convergence
	w-$\lim_{z\downarrow 0}\P^\dag_z= \P_0$
hold for a non-degenerate weak limit $\P_0$ and, if so, how can $\P_0$ be characterized?
Both tasks were solved subsequently by Bertoin and Caballero \cite{BC}, Bertoin and Yor \cite{BY},
 Caballero and Chaumont \cite{CC} and Chaumont et al. \cite{CKPR}.
We should mention that in cases (2) and (3) the first problem at zero contains the second problem at zero, since
 the first hitting time $T_0$ of zero tends to zero almost surely for initial condition $z$ tending to zero.
Hence, in cases (2) and (3) the solution was already given by Fitzsimmons \cite{F} and Rivero \cite{R1}, {\cite{R2}}
 so that case (1) can be assumed.
The main result is due to Caballero and Chaumont \cite[Theorem 2]{CC}: the weak limit
 w-$\lim_{z\downarrow 0}\P_z^\dag=\P_0$ exists if and only if the overshoot process
 \begin{align*}
		\xi_{T_x}-x,\;x\geq 0,\quad\text{with }\quad T_x:=\inf\{t\geq 0:\xi_t\geq x\},
  \end{align*}
 converges, as $x$ tends to $+\infty$, weakly towards the law of a finite random variable and a further
 technical condition on $\xi$ is satisfied. It was later shown in Chaumont et al. \cite{CKPR}
 that the additional assumption is superfluous.
The convergence of the overshoot process has an equivalent analytic reformulation.
According to Doney and Maller \cite[Theorem 8]{DonMal}, see also the discussion on page 1014 of Caballero
 and Chaumont \cite{CC}, it is equivalent to
 	\begin{align*}%\label{ass}
    		\xi\text{ is not arithmetic and }\begin{cases}
			\text{ either } 0<\EE(\xi_1)\leq \EE(|\xi_1|)<\infty,\\
			\text{ or } \EE(|\xi_1|)<\infty, \EE(\xi_1)=0\text{ and } J<\infty,
		\end{cases}
	\end{align*}
	where
	\begin{align*}
		J:=\int_1^\infty \frac{x\Pi((x,\infty))}
              {1+\int_0^x\int_y^\infty \Pi((-\infty,-z))\,\dd z\, \dd y}
                \,\dd x,
	 \end{align*}
 and $\Pi$ denotes the L\'evy measure of $\xi$.
Finally, let us mention that the construction of $\P_0$ of Caballero and Chaumont \cite{CC} is based
 on involved L\'evy process theory and not easily accessible.
A simpler construction of $\P_0$ was given in Bertoin and Savov \cite{BS}
 via L\'evy processes indexed by $\R$.

\medskip

We will show below how to use jump type SDEs to construct the extensions in all cases.
This comes directly from Lamperti's representation which, surprisingly, offers more insight
 if it is reformulated via jump type SDEs. Furthermore, it seems reasonable that more applications, in particular for real-valued self-similiar Markov processes, can be deduced via the powerful methods of stochastic analysis.

\subsection{A motivating example}\label{sec:ex}
Let us start with a simple motivation by considering the only pssMp of self-similarity index $1$ with
 continuous sample paths. If $\xi$ is a Brownian motion with drift, say
 $\xi_t=\big(\delta-2\big)t + 2 B_t$, $t\geq 0$, $\delta\in\R$, then the Lamperti transformed pssMp is $Z^\dag$, where $Z$ is the non-negative self-similar Markov process of index $1$ defined by
 \begin{align}\label{sb}
		\dd Z_t = \delta \dd t+2\sqrt{Z_t}\,\dd B_t.
 \end{align}
 For classical results on squared Bessel processes we refer to Chapter XI of Revuz and Yor \cite{RY};
 recall in particular that for $\delta\geq 0$ and initial condition $z\geq 0$, (\ref{sb}) has a unique strong solution defined for $t\geq 0$ and that $Z$ hits zero in finite time
 if and only if $\delta<2$.
In case $\delta< 0$ solutions exist only up to first hitting zero and are then extended by $0$.
The starting point of our analysis is the simple identity $\log \EE(\ee^{\xi_1})=\delta$,
 so that (\ref{sb}) can be written as
	\begin{align}\label{sb2}
		\dd Z_t=\log \EE(\ee^{\xi_1})\,\dd t+2 \sqrt{Z_t}\dd B_t.
	\end{align}
Thus, precisely if $\log \EE(\ee^{\xi_1})>0$ it has unique strong solutions, defined for all $t\geq 0$
 and for all non-negative initial conditions, which do not have $0$ as a trap.
This gives the following explanation for the two problems in case of leaving zero continuously:
\begin{itemize}
	\item If $Z^\dag$ hits zero in finite time and $\log \EE(\ee^{\xi_1})>0$,
           then the unique recurrent self-similar extension {of $Z^\dag$} that leaves zero continuously
           is nothing else but the unique strong solution $Z$ of (\ref{sb2}) for $t\geq 0$.
	\item If $Z^\dag$ does not hit zero in finite time (equivalently $\log \EE(\ee^{\xi_1})\geq 2$),
          then w-$\lim_{z\downarrow 0}\P_z=\P_0$ and $\P_0$ is nothing else but
          the unique strong solution of the SDE \eqref{sb2} started at zero (see for instance Stroock and Varadhan \cite[Corollary 11.1.5]{StrVar} for the convergence).
\end{itemize}
	For the particular example of the squared Bessel process we found naturally the condition $\log \EE(\ee^{\xi_1})>0$ but not Condition (\ref{cramer}). This can be seen {here from} the convexity of the Laplace exponent defined by
 $\psi(\lambda):=\log \EE(\ee^{\lambda \xi_1})$  for $\lambda\in\R:$
recalling that
 \begin{align}\label{a}
 	Z^\dag \text{ hits zero in finite time}\quad \Longleftrightarrow \quad \xi\text{ drifts to }-\infty\quad \Longleftrightarrow \quad \psi'(0+)=\EE(\xi_1)<0,
 \end{align}
the convexity of $\psi$ implies the equivalence
	\begin{align}\label{a_added}
		\psi(1)= \log \EE(\ee^{\xi_1}) >0
          \quad\Longleftrightarrow \quad
       \text{Condition }   (\ref{cramer}) \text{ holds with $a=1$ and $q=0$,}
	\end{align}
if $\xi$ drifts to $-\infty$.
{Note that \eqref{a} and \eqref{a_added} hold more generally for spectrally negative L\'evy processes
 possibly with killing, see, e.g. Kyprianou \cite[page 81]{K}.}

\subsection{A brief reminder of jump type SDEs for branching processes}\label{CSBPsection}

{To motivate our approach let us recall a development
 in the theory of continuous state branching processes (CSBPs),
 i.e. $\R_+$-valued c\`{a}dl\`{a}g strong Markov processes that satisfy the branching property
 $\P_x\ast \P_y=\P_{x+y}$, $x,y\geq 0$.
Lamperti \cite{L2} proved that any CSBP $Y$ can be represented as
	\begin{align*}
		{Y_t}=\xi_{\bar\tau(t)},\quad t\geq 0,\quad \text{ where }\quad	
           \bar\tau(t):=\inf\{t\geq 0: \bar I_s\geq t \}
             \quad\text{ and }\quad \bar I_t:=\int_0^{\min(t, \bar T_0)} \frac{1}{\xi_s}\,\dd s,
	\end{align*}
    where $\xi$ is a L\'evy process with no negative jumps started at $y\geq 0$ and
    $\bar T_0$ is the first time that $\xi$ hits $0$.
By Volkonskii's formula (see, e.g. Williams \cite[Section III.38]{W}) this time-change representation
 is equivalent to the statement that the infinitesimal generator of {$Y$} has the form
 $\mathcal A_{ Y} f(y)=y \mathcal A_\xi f(y)$ {for sufficiently smooth functions $f$}, where
	\begin{align*}
		\mathcal A_\xi f(y)=\gamma f'(y)+\frac{\sigma^2}{2}f''(y)
           +\int_0^\infty (f(y+u)-f(y)-f'(y)u\mathbf 1_{\{u\leq 1\}})\Pi(\dd u),
	\end{align*}
 is the infinitesimal generator of the L\'evy process $\xi$ with L\'evy triplet $(\gamma,\sigma^2,\Pi)$.
Recall that $\gamma\in \R$, $\sigma>0$ and $\Pi$ is a deterministic measure on $\R$ satisfying
 $\int_{\R}\min(1,u^2)\Pi(\dd u)<\infty$.
From the generator $\mathcal A_{Y}$ it is not hard to guess that ${Y}$ is a (weak) solution
 to the jump type SDE
	\begin{align}\label{de}
		\begin{split}
			{Y}_t&=y+\gamma\int_0^t{Y}_s\,\dd s+\sigma\int_0^t \sqrt{{Y}_s}\dd B_s\\
		&\quad+\int_0^t\int_0^{{Y}_{s-}}\int_0^1u \,(\mathcal{N-N'})(\dd s,\dd r,\dd u)
         +\int_0^t\int_0^{{Y}_{s-}}\int_1^\infty u\, \mathcal{N}(\dd s, \dd r,\dd u),
         \quad t\geq 0,
		\end{split}
	\end{align}
where $\mathcal N$ is a Poisson random measure independent of the Brownian motion $B$ having
 intensity measure $\mathcal N'(\dd s,\dd r,\dd u)=\dd s\otimes \dd r\otimes \Pi(\dd u)$. Interestingly, $Y$ is not only a weak solution to (\ref{de}) but actually a strong solution since pathwise uniqueness holds. This non-trivial additional fact, due to  Dawson and Li \cite[Section 5]{DL_affine}, suggests that the jump type SDE (\ref{de}) might be a strong tool to study CSBPs.
For further reading on Lamperti's transformation for CSBPs we refer to the overview article Caballero et al. \cite{CLU}.

\section{Results}	
\subsection{A jump type SDE approach to positive self-similar Markov processes }\label{sec:jumpdif}
Mirroring the jump type SDE approach to CSBPs we now derive a jump type SDE for pssMps of self-similarity index $1$.
This can be seen as natural generalization of (\ref{sb2}).
According to the discussion for CSBPs, using Volkonskii's formula and Lamperti's transformation (\ref{LT})
 with $a=1$, one gets the infinitesimal generator identity
  $\mathcal A_Z f(z)=\frac{1}{z} \mathcal A_{\ee^\xi} f(z)$ for sufficiently smooth functions {$f$}.
In order to derive from this a jump type SDE we proceed in two steps. {Recall that}
  the L\'evy-It\=o representation of $\xi$ takes the form
 \begin{align}\label{li}
   \xi_t=\gamma t + \sigma B_t+\int_0^t\int_{|u|\leq 1} u \,(\mathcal{N}_0-\mathcal{N}_0')(\dd s,\dd u)
                      +\int_0^t\int_{|u|> 1} u\, \mathcal{N}_0\,(\dd s,\dd u),\quad t\geq 0,
 \end{align} where $(\gamma,\sigma^2,\Pi)$ is the L\'evy triplet of $\xi$,
 $B$ is a standard Wiener process and $\mathcal{N}_0$ is an independent Poisson random measure
  on $(0,\infty)\times \R$ with intensity measure $\mathcal{N}_0'(\dd s,\dd u)=\dd s\otimes \Pi(\dd u)$.
  {First, applying} It\=o's formula to $z\ee^{\xi_t}$
 and afterwards including the generator correction $1/z$ similarly as the correction $y$ in  $\mathcal A_{ Y} f(y)=y \mathcal A_\xi f(y)$ is added to (\ref{de}), our {Ansatz - which is verified below in Proposition \ref{prop:reversed} - }is the jump type SDE
	 \begin{align}\label{ab}
		  \begin{split}
		Z_t&=z+\Big(\gamma+\frac{\sigma^2}{2}+\int_{|u|\leq 1}(\ee^u-1-u)\Pi(\dd u)\Big)t
              + \sigma\int_0^{t } \sqrt{Z}_s \dd B_s\\
		&\quad + \int_0^{t }\int_0^{\infty}\int_{|u|\leq 1}\mathbf 1_{\{r Z_{s-}\leq 1\}}Z_{s-}(\ee^u-1)
                             (\mathcal N-\mathcal N')(\dd s,\dd r,\dd u)\\
		&\quad+\int_0^{t}\int_0^{\infty}\int_{|u|>1}\mathbf 1_{\{r Z_{s-}\leq 1\}}Z_{s-}(\ee^u-1)
                 \mathcal N(\dd s,\dd r,\dd u)
		  \end{split}
	\end{align}
  which, as in Lamperti's transformation \eqref{LT}, is a priori restricted to $t\leq T_0$,
  since the constant drift might be negative.	
Here, $B$ is a standard Wiener process and $\mathcal N$ is an independent Poisson random measure
 on $(0,\infty)\times (0,\infty)\times \R$ with intensity measure
 $\mathcal N'(\dd s,\dd r,\dd u)=\dd s\otimes \dd r \otimes \,\Pi(\dd u)$.
On  first view there is not too much to learn from this reformulation
 of Lamperti's transformation (\ref{LT}).
However, if we additionally assume that $E(\ee^{\xi_1})<\infty$, then one can add and subtract the compensation
 of large jumps to (\ref{ab}), and we get
 	\begin{align*}
		Z_t&=z+\left(\gamma + \frac{\sigma^2}{2}+\int_\R(\ee^u-1-u\mathbf 1_{\{\vert u\vert\leq 1\}})
                \,\Pi(\dd u) \right)t
               + \sigma\int_0^{t } \sqrt{Z_s} \dd B_s\\
           &\quad+\int_0^{t }\int_0^{\infty}\int_{\R} \mathbf 1_{\{r Z_{s-}\leq 1\}}Z_{s-}[\ee^u-1]
               (\mathcal N-\mathcal N')(\dd s,\dd r,\dd u)
	\end{align*}
	for $t\leq T_0$.
 Simplifying via the L\'evy-Khintchin formula (more precisely the extension in Theorem 25.17 of Sato \cite{Sat}) we obtain
	\begin{align}\label{el}
     \begin{split}
		Z_t&=z+\big(\log \EE\big(\ee^{\xi_1}\big)\big)t + \sigma\int_0^{t } \sqrt{Z^{}_s} \dd B_s\\
     &\quad      +\int_0^{t }\int_0^{\infty}\int_{\R} \mathbf 1_{\{r Z_{s-}\leq 1\}}Z_{s-}[\ee^u-1]
               (\mathcal N-\mathcal N')(\dd s,\dd r,\dd u)
     \end{split}
	\end{align}
	for $t\leq T_0$. Let us now explain how to modify the Ansatz if we allow $\xi$ to be killed, i.e. set $\xi$ to
 the cemetery state ${-\infty}$ after {an} independent exponential lifetime $\zeta$. After taking the exponential for the first step of the derivation, this is equivalent to including jumps to zero at an independent exponential rate. As before, the time-change is included via an extra integral so that Equation (\ref{el}) transforms into
	\begin{align*}
     \begin{split}
		Z_t&=z+\big(\log \EE\big(\ee^{\xi_1}\big)\big)t + \sigma\int_0^{t } \sqrt{Z^{}_s} \dd B_s-\int_0^t\int_0^\infty\mathbf 1_{\{r Z_{s-}\leq 1\}} Z_{s-}\mathcal M(\dd s,\dd r)
\\
           &\quad+\int_0^{t }\int_0^{\infty}\int_{\R} \mathbf 1_{\{r Z_{s-}\leq 1\}}Z_{s-}[\ee^u-1]
               (\mathcal N-\mathcal N')(\dd s,\dd r,\dd u)
                  \end{split}
	\end{align*}
	for $t\leq T_0$, where $\mathcal M$ is an independent Poisson random measure on $(0,\infty)\times (0,\infty)$ with intensity measure $\mathcal M'(\dd s,\dd r)=q\dd s\otimes \dd r$ and $q\geq 0$ is the killing rate {(i.e. $\P(\zeta>1)=\ee^{-q}$).} Adding and subtracting the compensation gives
	\begin{align*}
     \begin{split}
		Z_t&=z+\big(\log \EE\big(\ee^{\xi_1}\big)-q\big)t + \sigma\int_0^{t } \sqrt{Z^{}_s} \dd B_s         -\int_0^t\int_0^\infty\mathbf 1_{\{r Z_{s-}\leq 1\}} Z_{s-}(\mathcal M-\mathcal M')(\dd s,\dd r)
\\
           &\quad+\int_0^{t }\int_0^{\infty}\int_{\R} \mathbf 1_{\{r Z_{s-}\leq 1\}}Z_{s-}[\ee^u-1]
               (\mathcal N-\mathcal N')(\dd s,\dd r,\dd u)
     \end{split}
	\end{align*}
	which is simplified to
	\begin{align}\label{eqn}
     \begin{split}
		Z_t&=z+\big(\log \EE\big(\ee^{\xi_1};\zeta>1\big)\big)t + \sigma\int_0^{t } \sqrt{Z_s} \dd B_s\\
           &\quad-\int_0^t\int_0^\infty\mathbf 1_{\{r Z_{s-}\leq 1\}} Z_{s-}(\mathcal M-\mathcal M')(\dd s,\dd r)\\
           &\quad+\int_0^{t }\int_0^{\infty}\int_{\R} \mathbf 1_{\{r Z_{s-}\leq 1\}}Z_{s-}[\ee^u-1]
               (\mathcal N-\mathcal N')(\dd s,\dd r,\dd u)\\
     \end{split}
	\end{align}
	for $t\leq T_0$.	
	 This is our generalization of (\ref{sb2}) which has a surprising advantage compared
 to Lamperti's original representation {\eqref{LT}}:

\begin{rem}\label{remark}
The Lamperti transformation \eqref{LT} only works for $t\leq T_0$, since, due to the time-change,
 the infinite time-horizon $[0,\infty)$ of $\xi$ is transformed into the {(possibly finite)} time-horizon $[0,T_0)$.
Nonetheless, if $0<\log \EE\big(\ee^{\xi_1};\zeta>1\big)<\infty$ then, due to the positive constant drift, the jump type SDE (\ref{eqn}) is not restricted to $t\leq T_0$
 but can be defined for all $t\geq 0$. This is not surprising since we already noted in Section \ref{sec:ex} that
 $\xi$ drifting to $-\infty$
 implies the equivalence
 \begin{align}\label{new1}
	\log \EE\big(\ee^{\xi_1};\zeta>1\big)>0\quad\Longleftrightarrow\quad  (\ref{cramer}) \text{ holds with } a=1
\end{align}
if the occurring quantities are finite. Hence, if we can show that (\ref{eqn})  is a "good" SDE, then the problems at zero for the corresponding pssMp can be explained as simple as for the squared Bessel process.
\end{rem}

To formulate our results, let us start with a rigorous definition of solutions that are constructed on a stochastic basis $(\Omega,\cG,(\cG_t)_{t\geq 0},P)$ satisfying the usual conditions.
We will use \textit{weak solutions}, i.e. $(\mathcal G_t)_{t\geq 0}$ adapted stochastic processes
 $(Z_t)_{t\geq 0}$ with almost surely c\`{a}dl\`{a}g paths that satisfy the SDE (\ref{eqn}) almost surely
 with some $(\cG_t)$-standard Wiener process $B$ and some independent $(\cG_t)$-Poisson random measure
 $\cN$ on $(0,\infty)\times (0,\infty)\times\R$ that has intensity measure
 $\cN'(\dd s,\dd r,\dd u) = \dd s\otimes \dd r\otimes \Pi(\dd u)$.
If additionally $Z$ is adapted to the augmented filtration generated by $B$ and $\mathcal N$, then $Z$ is said
 to be a \textit{strong solution}.
We say that \textit{pathwise uniqueness} holds for the SDE (\ref{eqn}) if for any two weak solutions
 $Z^1$ and $Z^2$ defined on the same probability space with the same standard Wiener process
 and Poisson random measure,  they are indistinguishable. Weak solutions for the SDE (\ref{eqn}) \textit{up to a first hitting time} are defined via the
 localized version of the corresponding martingale problem as in Chapter 4.6 of Ethier and Kurtz \cite{EK}.

\begin{thm}\label{thm:1}
Suppose that $\xi$ is a spectrally negative L\'evy process with L\'evy triplet $(\gamma,\sigma^2,\Pi)$ {killed at rate $q\geq 0$}. If $\log \EE\big(\ee^{\xi_1};\zeta>1\big)>0$, then
 \begin{itemize}
 	\item[(a)]  for any initial condition $z\geq 0$ there is
                a pathwise unique non-negative strong solution $(Z_t^{(z)})_{t\geq 0}$
                of the SDE (\ref{eqn}) associated to $\xi$ {killed at rate $q$},
    \item[(b)] the solutions $\{Z^{(z)},{z\geq 0}\}$ define a non-negative self-similar Markov process of self-similarity index $1$,
 {\item[(c)] the solutions $\{(Z^{(z)})^\dag,z\geq 0\}$ define a positive self-similar Markov process of index $1$ and its Lamperti transformed L\'evy process has L\'evy triplet
                 $(\gamma,\sigma^2,\Pi)$} {and is killed at rate $q$.}
 \end{itemize}
 \end{thm}

Recently, in Fu and Li \cite{FL}, Li and Mytnik \cite{LM} and Dawson and Li \cite{DL},
 the pathwise uniqueness problem for Poisson driven jump type SDEs has been studied.
Our SDE (\ref{eqn}) does not fit precisely into those frameworks because of missing monotonicity
 properties and the explosive jump rate at zero. However, we can adapt some of their
 arguments in our special case.
At the beginning of Section \ref{sec:proofs} we briefly explain why we assume that $\xi$ is spectrally negative and
 we explain why we conjecture that the minimal assumption $\EE(\ee^{\xi_1})<\infty$ should be enough for
 deriving the SDE (\ref{eqn}).

\subsection{Applications to the problems at zero}

	Before stating our contributions to the problems at zero let us discuss how to reduce pssMps of index
 $a>0$ to those of index $1$. For this sake we use the following fact which follows readily from the definition
 of Lamperti's transformation (see for instance Bertoin and Yor \cite[page 203]{BY3}):
	\begin{align}\label{b}\begin{split}
		&\quad Z^\dag\text{ is a pssMp of index } a \text{ with associated L\'evy process} \ \xi\\
		&\Longleftrightarrow \quad X^\dag:=(Z^\dag)^{1/a}\text{ is a pssMp of index }1
      \text{ with associated L\'evy process} \ \frac{1}{a}\xi.\end{split}
	\end{align}
Hence, in order to study recurrent self-similar extensions $(Z_t)_{t\geq 0}$ of $Z^\dag$ it suffices
 to find an extension $(X_t)_{t\geq 0}$ of $X^\dag$, since then
 $Z:=X^a$ is an extension of $Z^\dag$.
Since the Lamperti transformed L\'evy process for $X^\dag$ is explicitly given by $\frac{1}{a}\xi$,
 this can be done via Theorem \ref{thm:1} provided that the Lamperti transformed L\'evy process $\xi$
 is known for $Z^\dag$.
\smallskip

We can now state a new characterization of the unique recurrent self-similar extension of a pssMp
 that leaves zero continuously in the spirit of Fitzsimmons \cite{F} and Rivero \cite{R2}.
		
\begin{thm}\label{thm:3}
Suppose that $\xi$ is a spectrally negative L\'evy process {killed at rate $q\geq 0$} and the Cram\'er type condition (\ref{cramer}) holds. Further, let $(X_t)_{t\geq 0}$ be the unique strong solution of (\ref{eqn})  associated to $\frac{1}{a}\xi$ killed at rate $q$ with initial condition $z^{1/a}\geq 0$. If $Z^{(z)}:=X^a$, then
 \begin{itemize}
 	\item[(a)] $\{(Z^{(z)})^\dag,z\geq 0\}$ is a pssMp of self-similarity index $a$ with Lamperti transformed L\'evy process
               $\xi$ killed at rate $q$,
	\item[(b)] $\{Z^{(z)},z\geq 0\}$ is the unique recurrent {self-similar} extension of $\{(Z^{(z)})^\dag,z\geq 0\}$   that leaves zero continuously.
 \end{itemize}
 \end{thm}
Note that under the assumptions of Theorem \ref{thm:3}, the L\'evy process $\xi$ drifts to $-\infty$
 so that $\log \EE(\ee^{\frac{1}{a}\xi_1};{\zeta>1})>0$ by (\ref{new1}).
Hence, the SDE (\ref{eqn}) associated to $\frac{1}{a}\xi$ {killed at rate} $q$ has a strictly positive constant drift
 and the result is a consequence of {\eqref{b} and} Theorem \ref{thm:1}.
As usual in the study of pssMps, Theorem \ref{thm:3} is only applicable for pssMps for which
 $\xi$ is known.
For some examples of pssMps for which the {L\'evy triplet of the Lamperti transformed L\'evy process}
 $\xi$ can be calculated explicitly, we refer to Caballero and Chaumont \cite{CC2}.

 \smallskip

Finally, we utilize the fact that the SDE (\ref{eqn}) can be issued from zero once the constant drift is strictly positive.
This gives a characterization and convergence statement for pssMps started from zero via stochastic calculus.

\begin{thm}\label{thm:4}
Let $a>0$ and suppose that $\xi$ is a spectrally negative L\'evy process killed at rate $q\geq 0$ {and that} either
\begin{itemize}
	\item $\xi$ does not drift to $-\infty$,
	\item or $\xi$ drifts to $-\infty$ and Condition (\ref{cramer}) holds.
\end{itemize}
 If $\{\P_z^\dag,z\geq 0\}$  denotes the law of the {pssMp of self-similarity index $a$ with Lamperti transformed
 L\'evy process $\xi$ killed at rate $q$, then
 \[
   \text{w-}\lim_{z\downarrow 0}\P_z^\dag= \P_0 \quad \text{in the Skorokhod topology},
 \]
 where $\P_0$ is the law of $Z^{(0)}:=X^a$ and $X$ is the unique strong solution of
 the SDE (\ref{eqn}) associated to $\frac{1}{a}\xi$ killed at rate $q$ with initial condition $X_0=0$.}
\end{thm}

Theorems \ref{thm:3} and \ref{thm:4} were motivated by results on CSBPs with state-dependent immigration in Berestycki et al. \cite{BerDorMytZam}. In their context the pssMp was already defined by a jump type SDE and the extensions could be constructed more directly.

\begin{rem}
In the setting of Theorem \ref{thm:1} with $q=0$, applying It\={o}'s formula to the SDE (\ref{eqn}),
 one can check that the infinitesimal generator of the pssMp of index 1 associated to $\xi$
 takes the form
 \begin{align*}
 \mathcal A f(z)
   &=(\log \EE(\ee^{\xi_1}))f'(z) + \frac{\sigma^2}{2}zf''(z)
           +\frac{1}{z}\int_\R \big(f(z\ee^u)-f(z)-zf'(z)(\ee^u-1)\big)\Pi(\dd u)\\
   &=\left(\gamma + \frac{\sigma^2}{2}\right)f'(z) + \frac{\sigma^2}{2}zf''(z)
      +\frac{1}{z}\int_\R \big(f(z\ee^u)-f(z)
               - zf'(z)u\mathbf 1_{\{\vert u\vert\leq 1\}}\big)\Pi(\dd u)
 \end{align*}
 for $z>0$ and for appropriate smooth test-functions $f$.
The drift coefficient is $\gamma+\frac{\sigma^2}{2}$ (and not $\gamma$)
 which is a misprint in formula (6.5) of Lamperti \cite[Theorem 6.1]{L}.
\end{rem}%%%%%%%%%%%%%%%%%%%%%%%%%%%%%%%%%%%%%%%%%%%%%%%%%%%%%%%%%%%%%%%%%%%%
%%%%%%%%%%%%%%%%%%%%%%%%%%%%%%%%%%%%%%%%%%%%%%%%%%%%%%%%

\section{Proofs}\label{sec:proofs}

From now on we will work on a stochastic basis $(\Omega,\cG,(\cG_t)_{t\geq 0},P)$ satisfying
 the usual conditions and we will denote the underlying probability measure
 and the expectation with respect to it by $P$ and $\E$, respectively, instead of $\P$ and $\EE$
 (in contrast to the introduction, but without confusion).
Further, to simplify the notations (in contrast to the formulation of Theorems \ref{thm:1}, \ref{thm:3} and \ref{thm:4})
 we will write $Z$ instead of $Z^{(z)}$, i.e. we omit denoting the initial value $z\geq 0$.
\smallskip

Before turning to the proofs we briefly discuss the difficulties of the SDE (\ref{eqn}).
The pathwise uniqueness proof is based on a Yamada-Watanabe type argument.
This argument has been reinvented for jump-type SDEs by Fu and Li \cite{FL} motivating
our approach as well even though the coefficients of the SDE (\ref{eqn}) do not satisfy the standard requirements.
To understand the additional difficulty of proving pathwise uniqueness for (\ref{eqn}) we stress two issues:\\}
 For the first, one has to consider carefully the large
  jumps, i.e. those jumps which come from
  an atom $(s,r,u)$ of $\mathcal N$ with $\vert u\vert$ larger than some fixed $M>0$.
In the setting of Fu and Li \cite{FL}, it was possible to truncate the jump measure at infinity.
Pathwise uniqueness for their jump type SDE with truncated intensity measure
 could be proved via second moment arguments and large jumps
 were added by interlacing since their occurrence does not accumulate.
Unfortunately, the large jumps accumulate for our SDE (\ref{eqn})
 when solutions approach zero.
Hence, the interlacement procedure cannot be fully used in our case so that estimates need to
 be carried out on the full L\'evy measure. \\
The second issue directly occurs in the estimates of the Yamada-Watanabe argument (or local time argument) for the pathwise uniqueness proof. The adaptation of the Yamada-Watanabe argument is relatively easy if for two solutions $X$ and $Y$
 \begin{align}\label{mono}
 	(X_{s-}-Y_{s-})\Delta (X_{s}-Y_{s})\geq 0,\quad s>0,
 \end{align}
 i.e. the larger solution has the larger jumps which equivalently says that the integrand of the Poissonian integral is non-decreasing. This monotonicity property is for instance verified for the CSBPs discussed in Section \ref{CSBPsection} and is a main assumption in Fu and Li \cite{FL} and Li and Mytnik \cite{LM}. Modifying some arguments, uniqueness can still be proved if the smaller solution does not exceed the larger solution by a jump, i.e. after a jump the difference $X-Y$ does not change sign. This is the key observation allowing to prove pathwise uniqueness for generalized Fleming-Viot processes in Dawson and Li \cite{DL}.  In fact, this is the precise origin of our hypothesis $\xi$ being spectrally negative. The monotonicity property (\ref{mono}) always fails due to the jump rate $1/Z_{s-}$ in Equation (\ref{eqn}) but a change of sign for the difference could occur only for jumps corresponding to $u>0$.
\begin{rem}\label{remark_cond}
	Apart from the proof of pathwise uniqueness given in Proposition \ref{prop:uniqueness}, all of our arguments hold equally if $\E(\ee^{2\xi_1})<\infty$ which is trivially fulfilled if $\xi$ is spectrally negative.
	To simplify a possible later extension, we kept the proofs in this larger generality,
 but behind always supposing that $\xi$ is spectrally negative.	
\end{rem}
\begin{rem}
	In the meantime the pathwise uniqueness has been generalized by Li and Pu \cite{LP} to
more general jump-type SDEs containing the SDE \eqref{eqn}.
\end{rem}
We now start with the proof of Theorem \ref{thm:1} for which we first prove the pathwise uniqueness statement and then construct a strong solution via approximation.
To prove pathwise uniqueness for the SDE (\ref{eqn}) the non-Lipschitz integrand $x\mapsto \sqrt{x}$ of
 the Brownian part forces us to use Yamada-Watanabe type arguments going back to \cite{YW1} and \cite{YW2}.\\
Let us start with some notations. Suppose $a_0=1$ and $0<...<a_k<a_{k-1}<...\leq a_0$, $k\in\N$, are such that
 $\int_{a_k}^{a_{k-1}}\frac{1}{z}\dd z=k$ for all $k\geq 1$.
For completeness, we note that $a_k=\ee^{-k}a_{k-1}$,  which yields that
	\begin{align*}
		a_k = \ee^{-\frac{k(k+1)}{2}},\qquad k\in\N,
	\end{align*}
 and hence $\lim_{k\to\infty} a_k = 0$. In what follows we will not use the explicit form of $(a_k)_{k\in\N}$.
For all $k\in\N$, let $\psi_k:\R\to\R_+$ be a non-negative continuous function with support in $(a_k,a_{k-1})$
 satisfying $\int_{a_k}^{a_{k-1}}\psi_k(x)\dd x=1$ and $\psi_k(x)\leq \frac{2}{kx}$, $x>0$.
Next, let us define
 \begin{align*}
		\phi_k(z):=\int_0^{|z|}\int_0^y \psi_k(x)\,\dd x\,\dd y,\quad z\in\R.
 \end{align*}
Then it is apparent that $\phi_k$ is even, twice continuously differentiable,
	\begin{align}\label{b0}
     		\lim_{k\to\infty}\phi_k(z) = |z|,\quad z\in\R,
	\end{align}
	and the sequence $(\phi_k)_{k\in\N}$ is non-decreasing. Furthermore, for fixed $k\in\N$ and $x,y\in\R_+$, we have
	\begin{align}\label{b1}
		\phi_k''(x-y)[\sqrt{x}-\sqrt{y}]^2=\psi_k(\vert x-y\vert)[\sqrt{x}-\sqrt{y}]^2\leq \frac{2[\sqrt{x}-\sqrt{y}]^2}{k|x-y|}\leq \frac{2}{k}.
	\end{align}
Supposing that there are two {non-negative} weak solutions $X$ and $Y$ of the SDE (\ref{eqn}) on the same stochastic basis
 $(\Omega,\cG,(\cG_t)_{t\geq 0},P)$
 with the same standard Wiener process and Poisson random measures {such that $X_0=Y_0$},
 our aim is to estimate $\E(\phi_k(X_t-Y_t))$ so that via Fatou's lemma and (\ref{b0}) an estimate for $\E(|X_t-Y_t|)$
 can be derived.
In the following we denote by
	\begin{align*}
		\Delta_h f(x):=f(x+h)-f(x)\quad\text{and}\quad D_hf(x):=\Delta_h f(x)-f'(x)h
	\end{align*}
	if the right-hand sides are meaningful and additionally abbreviate
	\begin{align*}
		{\bf g}(x,r,u)&=\mathbf 1_{\{rx\leq 1\}}x(\ee^u-1),\qquad x\geq 0,\;r\geq 0,\;u\in\R,\\
		{\bf h}(x,r)&=-\mathbf 1_{\{rx\leq 1\}}x,\qquad x\geq 0,\;r\geq 0.
	\end{align*}
	From now on we will also abbreviate $a\wedge b:=\min(a,b)$ and $a\vee b:=\max(a,b)$ for $a,b\in\R$. Let us start with a simple lemma {about} convergence to infinity of a suitable sequence of stopping times that are used to ensure that the appearing stopped local martingales are martingales.

\begin{lemma}\label{lll}
	Suppose that $Z$ is a non-negative weak solution of the SDE (\ref{eqn}) and let $\tau_m:=\inf\{t\geq 0: Z_t\geq m\}$ for $m\in \N$. Then the sequence $(\tau_m)_{m\in\N}$ is increasing and tends to infinity almost surely.
\end{lemma}

\begin{proof}
{First let} us assume that $q=0$.
Note that $\tau_m$ is increasing in $m$ so that $\lim_{m\to \infty}\tau_m$ exists almost surely.	
By Proposition 2.3 in Fu and Li \cite{FL}, we find that $\tau_m$ indeed converges to infinity almost surely as
 $m\to\infty$.
To check that we are allowed to use this proposition, their conditions (2.a) and (2.b) need to be verified.
In order to do so, let
 \begin{align*}
  & b(x):=\begin{cases}
           \log E(\ee^{\xi_1}) & \text{if $x>0$}\\
                             0 & \text{if $x\leq 0$}
          \end{cases} \quad\text{and}
  \quad \sigma(x):=\begin{cases}
                \sigma\sqrt{x} & \text{if $x>0$}\\
                             0 & \text{if $x\leq 0$}
                \end{cases}
  \end{align*}
 $\mu_1:= 0$, $g_1:= 0$, $U_0:=(0,\infty)\times \R$, $\mu_0(\dd r,\dd u) := \dd r\otimes \Pi(\dd u)$
 and
 $  g_0(x,r,u) := {\bf g}(x,r,u)$.
Then condition (2.a) in Fu and Li \cite{FL} is satisfied with $K:=\log E(\ee^{\xi_1})$, and, for all $x>0$,
 \begin{align*}
 &  \sigma(x)^2 + \int_0^\infty\int_\R \Big[\vert g_0(x,r,u)\vert\wedge g_0(x,r,u)^2\Big]
                    \,\mu_0(\dd r,\dd u) \\
 &\quad \leq \sigma^2x+ \int_0^\infty\int_\R \mathbf 1_{\{xr\leq 1\}}
                                        x^2(\ee^u-1)^2\,\dd r\,\Pi(\dd u) \\
 &\quad\leq \sigma^2 x+x\int_\R(\ee^u-1)^2\,\dd r\,\Pi(\dd u),
  \end{align*}
 which implies that condition (2b) in Fu and Li \cite{FL} is also satisfied.
 Since we have to suppose that $\xi$ is spectrally negative, the proof can be extended easily to $q>0$ if the Poisson random measure $\mathcal M$ is chosen to be $\mathcal M(\dd s,\dd r)=\mathcal N(\dd s,\dd r,\{1\})$, with an additional atom of weight $q$ for $\Pi$ at $\{1\}$, and $\bf g$ is extended as $\mathbf{g} (x,r,1)=-\mathbf 1_{\{rx\leq 1\}}x$ for $ x,r\geq 0$.
 \end{proof}

We can now proceed to prove the pathwise uniqueness. The proof is a combination of the Yamada-Watanabe argument with the
 classical {Tanaka's formula approach} for pathwise uniqueness. To avoid the local time argument we use a dominated convergence
 argument.

\begin{proposition}\label{prop:uniqueness}
Suppose that $\xi$ is as in Theorem \ref{thm:1}, then pathwise uniqueness for non-negative solutions holds
 for the SDE (\ref{eqn}) {for $t\geq 0$}.
\end{proposition}

\begin{proof}
Suppose that ${(X_t)_{t\geq 0}}$ and ${(Y_t)_{t\geq 0}}$ are two non-negative weak solutions
 of the SDE (\ref{eqn}) on the same stochastic basis with the same standard Wiener process and Poisson random measures
 such that $X_0 = Y_0=z\geq 0$.
In the sequel we are going to show that $X$ and $Y$ are indistinguishable.
First, by It\=o's formula (see, e.g. {Di Nunno et al. \cite[Theorem 9.5]{NOP}} or Ikeda and Watanabe \cite[Chapter II, Theorem 5.1]{IW}) taking
 the difference of these solutions, the drifts cancel each other and it remains
\begin{align*}
  X_t-Y_t
     &=\sigma\int_0^t\Big[\sqrt{ X_{s}}-\sqrt{Y_{s}}\Big]\dd B_s
	                       \\
	                       &\quad+ \int_0^t\int_0^\infty\int_{\R} \big[{\bf g}(X_{s-},r,u)
                           - {\bf g}(Y_{s-},r,u)\big]
                             \mathcal{(N-N')}(\dd s,\dd r,\dd u)\\
                                                &\quad+ \int_0^t\int_0^\infty \big[{\bf h}(X_{s-},r)
                           - {\bf h}(Y_{s-},r)\big]
                             \mathcal{(M-M')}(\dd s,\dd r),\quad t\geq 0,
 \end{align*}
 since the noises are independent.
Next, for all $m\in\N$, let $\tau_m:=\inf\big\{t\geq 0\,:\, X_t \geq m\text{ or }  Y_t\geq m\big\}$ which  tends to infinity as $m\to\infty$ almost surely due to Lemma \ref{lll}.
Again by It\=o's formula, now applied to the semi-martingale $Z_t:= X_t- Y_t$, $t\geq 0$, we obtain
 \begin{align}\label{help_path_uni}
      \begin{split}
		\phi_k(Z_{t\wedge \tau_m})
	   &=\frac{\sigma^2}{2}\int_0^{t\wedge \tau_m}
	     \phi_k''(Z_{s})\Big[\sqrt{X_{s}}-\sqrt{ Y_{s}}\Big]^2 \,\dd s\\
	   &\quad+\int_0^{t\wedge \tau_m}\int_0^\infty\int_{\R}
	      D_{{\bf g}(X_{s},r,u)-{\bf g}( Y_{s},r,u)}
	    \phi_k(Z_{s})\,\dd s\,\dd r\,\Pi(\dd u)\\
	    	   &\quad+q\int_0^{t\wedge \tau_m}\int_0^\infty	      D_{{\bf h}(X_{s},r)-{\bf h}( Y_{s},r)}
	    \phi_k(Z_{s})\,\dd s\,\dd r\\
	   &\quad + \sigma \int_0^{t\wedge \tau_m} \phi_k'(Z_{s})
	     \Big[\sqrt{ X_{s}}-\sqrt{ Y_{s}}\Big] \,\dd B_s\\
	   &\quad + \int_0^{t\wedge \tau_m}\int_0^\infty \int_{\R}
             \Delta_{ {\bf g}(X_{s-},r,u) - {\bf g}( Y_{s-},r,u) }
                     \phi_k(Z_{s-})\mathcal{(N-N')}(\dd s,\dd r,\dd u)\\
                     	   &\quad + \int_0^{t\wedge \tau_m}\int_0^\infty
             \Delta_{ {\bf h}(X_{s-},r) - {\bf h}( Y_{s-},r) }
                     \phi_k(Z_{s-})\mathcal{(M-M')}(\dd s,\dd r)
      \end{split}
  \end{align}
 for $k, m\in\N$ and $t\geq 0$.
Let us first show that the last three integrals are proper martingales in $t$ with respect to the filtration $(\cG_t)$.
By Ikeda and Watanabe \cite[Chapter II, Proposition 2.4 and page 62]{IW}, for this it is enough to check that
 \begin{align}
	E\left( \int_0^{t\wedge \tau_m} (\phi_k'(Z_{s}))^2
    \Big(\sqrt{X_{s}}-\sqrt{Y_{s}}\Big)^2 \,\dd s\right)<\infty,\label{help8}\\
   E\left( \int_0^{t\wedge \tau_m}\!\!\int_0^\infty\!\! \int_{\R}
    \Big(\Delta_{ {\bf g}(X_{s},r,u) - {\bf g}( Y_{s},r,u) }
       \phi_k(Z_{s})\Big)^2\dd s\,\dd r\,\Pi(\dd u)\right)
      		<\infty,\label{help9}\\
		   E\left( q\int_0^{t\wedge \tau_m}\!\!\int_0^\infty\!\!
    \Big(\Delta_{ {\bf h}(X_{s},r) - {\bf h}( Y_{s},r) }
       \phi_k(Z_{s})\Big)^2\dd s\,\dd r\right)
      		<\infty,\label{help9b}
 \end{align}
 for $k,m\in\N$ and $t\geq 0$.
Using that $ X_t<m$ and $ Y_t<m$ for $0\leq t<\tau_m$, and that {$(\phi_k'(z))^2\leq 1$, $z\in\R$,}
(\ref{help8}) follows.
For \eqref{help9}, using the estimate $(\Delta_h\phi_k(x))^2\leq h^2$, $h\in\R$, $x\in\R$, we get the upper bound
	 \begin{align*}
	 	& E\Bigg( \int_0^{t\wedge \tau_m}\int_0^\infty \int_{\R}
	         ( {\bf g}\big( X_{s},r,u) - {\bf g}( Y_{s},r,u)\big)^2 \dd s\,\dd r\,\Pi(\dd u)\Bigg)\\
  		&\quad \leq E \left(  \int_0^{t\wedge \tau_m} \vert Z_{s}\vert \,\dd s \right)
                          \int_{\R} (\ee^u -1)^2\,\Pi(\dd u)<\infty,
	\end{align*}
since for all $x,y\in\R_+$,
		\begin{align*}
				&\quad \int_0^\infty\int_{\R} (\mathbf 1_{\{rx\leq 1\}}x
                         			    -\mathbf 1_{\{ry\leq 1\}}y)^2(\ee^u-1)^2\dd r\,\Pi(\dd u)\\
			&=
         		 \left(x+y-2\int_0^\infty \mathbf 1_{\{ry\leq 1\}}
                      \mathbf 1_{\{rx\leq 1\}}xy\,\dd r\right)\int_{\R} (\ee^u-1)^2\,\Pi(\dd u)\\
		    & =\left(x+y-2\min(x,y)\right)\int_{\R} (\ee^u-1)^2\,\Pi(\dd u)\\
			& =|x-y|\int_{\R} (\ee^u-1)^2\,\Pi(\dd u).
	\end{align*}
This shows  (\ref{help9})  since $\vert Z_s\vert\leq 2m$ for $0\leq s<\tau_m$.
Almost the identical estimate shows (\ref{help9b}).
Note also that the above calculations show that we had the right to use It\={o}'s formula for deriving \eqref{help_path_uni}.
Now we are in a position to carry out the Yamada-Watanabe argument for (\ref{help_path_uni}). The martingales vanish if we take expectations and the first integral can be estimated via (\ref{b1}) {having limit $0$ as $k\to\infty$
 for every fixed $t\geq 0$ and $m\in\N$.}
If we suppose additionally that, for any $t\geq0$ and $m\in\N$,
	 	\begin{align}
		\lim_{k\to \infty} \E\left( \int_0^{t\wedge \tau_m}\int_0^\infty\int_{\R} D_{{\bf g}(X_{s},r,u)-{\bf g}( Y_{s},r,u)}	    \phi_k\big(X_{s}-Y_s\big)\,\dd s\,\dd r\,\Pi(\dd u)	 \right)&=0,\label{l3}\\
				\lim_{k\to \infty} \E\left( q\int_0^{t\wedge \tau_m}\int_0^\infty D_{{\bf h}(X_{s},r)-{\bf h}( Y_{s},r)}	    \phi_k\big(X_{s}-Y_s\big)\,\dd s\,\dd r	 \right)&=0\label{l3b}
	\end{align}
	then  {\eqref{b0}}, \eqref{help_path_uni} and Fatou's lemma lead us to
	\begin{align*}
		0\leq E\big(|Z_{t\wedge \tau_m}|\big)&\leq\lim_{k\to\infty} E\big(\phi_k(Z_{t\wedge \tau_m})\big)=0
	\end{align*}
 for every $t\geq 0$ and $m\in \N$.
Using that $\tau_m$ tends to $\infty$ almost surely as $m\to\infty$, again by Fatou's lemma, we have
	\begin{align*}
		0\leq E\big(|Z_t|\big)\leq \lim_{m\to\infty} E\big(|Z_{t\wedge \tau_m}|\big)=0, \quad  t\geq 0.
	\end{align*}
Hence, the processes $\big\{ X_t, t\geq 0\big\}$ and
 $\big\{ Y_t, t\geq 0\big\}$ are modifications of one another.
Since both processes have right continuous sample paths they are also indistinguishable
 (see, e.g., Karatzas and Shreve \cite[Problem 1.1.5]{KS}).

\medskip

To finish the proof we still have to verify  (\ref{l3}) and (\ref{l3b}). Let us first define $\phi(z):=|z|$, $z\in\R$, and note that, with $\phi'(z)=\mathrm{sign}(z)$ for $ z\neq 0$,
	\begin{align}\label{bb}\begin{split}
	&\quad \int_0^\infty      D_{{\bf g}(x,r,u)-{\bf g}( y,r,u)}	    \phi\big(x-y\big)\,\dd r\\
	&=\mathbf 1_{\{x>y\}}\int_0^{1/x} \left(\big|(x-y)+(x-y)(\ee^u-1) \big|-\big|x-y\big|-\big(x-y\big)\big(\ee^u-1\big)\right)\dd r\\		
	&\quad+ \mathbf 1_{\{x>y\}}\int_{1/x}^{1/y} \left(\big|(x-y)-y(\ee^u-1) \big|-\big|x-y\big|{+}y\big(\ee^u-1\big)\right)\dd r\\
	&\quad+\mathbf 1_{\{x{\leq}y\}}\int_0^{1/y} \left(\big|(x-y)+(x-y)(\ee^u-1) \big|-\big|x-y\big|
           +\big(x-y\big)\big(\ee^u-1\big)\right)\dd r\\	
    &\quad+ \mathbf 1_{\{x{\leq}y\}}\int_{1/y}^{1/x} \left(\big|(x-y)+x(\ee^u-1) \big|-\big|x-y\big|+x\big(\ee^u-1\big)\right)\dd r\\
	&=0\end{split}
	\end{align}	
 for any $x,y\geq 0$ and any $u\leq 0$.
In fact, each of the four integrands is constant zero.
By Fubini's theorem (justified by the trivial fact that $|0|=0$) this shows that	
	\begin{align*}
		\E\left( \int_0^{t\wedge \tau_m}\int_0^\infty\int_{\R}      D_{{\bf g}(X_{s},r,u)-{\bf g}( Y_{s},r,u)}	    \phi\big(X_{s}-Y_s\big)\,\dd s\,\dd r\,\Pi(\dd u)	 \right)=0
	\end{align*}
	so that \eqref{l3} follows if we can show that
	\begin{align*}%\label{bbb}\begin{split}
		&\quad \lim_{k\to \infty} \E\left( \int_0^{t\wedge \tau_m}\int_0^\infty\int_{\R}      D_{{\bf g}(X_{s},r,u)-{\bf g}( Y_{s},r,u)}	    \phi_k\big(X_{s}-Y_s\big)\,\dd s\,\dd r\,\Pi(\dd u)	 \right)\\
		&= \E\left( \int_0^{t\wedge \tau_m}\int_0^\infty\int_{\R}      D_{{\bf g}(X_{s},r,u)-{\bf g}( Y_{s},r,u)}	    \phi\big(X_{s}-Y_s\big)\,\dd s\,\dd r\,\Pi(\dd u)	 \right).%\end{split}
	\end{align*}
By Fubini's theorem, separating the four cases as before {in \eqref{bb}}
 and integrating out $r$ we see that we can equally show that for any $t\geq 0$
 (the corresponding two cases for $X_s\leq Y_s$ are similar)
	\begin{align}\label{l1}\begin{split}
		&\lim_{k\to\infty} \E\Big(\int_0^t\int_{\R}\mathbf 1_{\{X_s>Y_s\}}\frac{1}{X_s}\Big(\phi_k(X_s-Y_s+(X_s-Y_s)(\ee^u-1))\\
		&\qquad-\phi_k(X_s-Y_s)-\phi_k'(X_s-Y_s)(X_s-Y_s)(\ee^u-1)\Big){\dd s\,\Pi(\dd u)}\Big)\\
		&= \E\Big(\int_0^t\int_{\R}\mathbf 1_{\{X_s>Y_s\}}\frac{1}{X_s}\Big(\big|X_s-Y_s+(X_s-Y_s)(\ee^u-1)\big|\\
		&\qquad-\big|X_s-Y_s\big|-(X_s-Y_s)(\ee^u-1)\Big){\dd s\,\Pi(\dd u)}\Big)\end{split}
	\end{align}
	and	
	\begin{align}\label{l2}\begin{split}
		&\lim_{k\to\infty} \E\Big(\int_0^t\int_{\R}
            \mathbf 1_{\{X_s>Y_s\}}\Big(\frac{1}{Y_s}-\frac{1}{X_s}\Big)\Big(\phi_k(X_s-Y_s-Y_s(\ee^u-1))\\
		&\qquad-\phi_k(X_s-Y_s)+\phi_k'(X_s-Y_s)Y_s(\ee^u-1)\Big){\dd s\,\Pi(\dd u)}\Big)\\
		&= \E\Big(\int_0^t\int_{\R}\mathbf 1_{\{X_s>Y_s\}}\left(\frac{1}{Y_s}-\frac{1}{X_s}\right)
          \Big(\big|X_s-Y_s-Y_s(\ee^u-1)\big|\\
  &\quad\quad         -\big|X_s-Y_s\big|+Y_s(\ee^u-1)\Big)           {\dd s\,\Pi(\dd u)}\Big).\end{split}
	\end{align}
Indeed, one can apply Fubini's theorem for \eqref{l1} and \eqref{l2}, since in all cases
 the integrands are non-negative which follows by the mean-value theorem representation for the remainder
 of Taylor's theorem and  $\phi_k''(x) = \psi_k(\vert x\vert)\geq 0$, $x\in\R$.
Due to the pointwise convergences $\lim_{k\to\infty}\phi_k(x)=|x|$, $x\in\R$, and $\lim_{k\to\infty}\phi_k'(x)=\textrm{sign}(x)$ for $x\neq 0$ it suffices to verify dominated convergence for both integrals. To prove (\ref{l1}), we split the large and small negative jumps to obtain
	\begin{align*}
		&\quad\mathbf 1_{\{X_s>Y_s\}}\frac{1}{X_s}\Big(\phi_k(X_s-Y_s+(X_s-Y_s)(\ee^u-1))\\
		&\phantom{\quad\mathbf 1_{\{X_s>Y_s\}}\frac{1}{X_s}\Big(}
              -\phi_k(X_s-Y_s)-\phi_k'(X_s-Y_s)(X_s-Y_s)(\ee^u-1)\Big)\\
		&\leq \mathbf 1_{\{X_s>Y_s\}}\mathbf 1_{\{u<-1\}}\frac{1}{X_s}\phi_k'(X_s-Y_s)(X_s-Y_s)(1-\ee^u)\\
		&\quad\quad +\mathbf 1_{\{X_s>Y_s\}}\mathbf 1_{\{-1\leq u\leq 0\}}\frac{1}{X_s}\frac{1}{2}\phi_k''(\varrho_s) (X_s-Y_s)^2 (\ee^u-1)^2
	\end{align*}
  for some $\varrho_s\in((X_s-Y_s)\ee^u,X_s-Y_s)$, where we used again the mean-value theorem representation
 for the remainder of Taylor's theorem and that $\phi_k(x)$ is increasing in $x\geq 0$ and $\phi_k'(x)\geq 0$, $x\in\R$.
 Using that $\phi_k'(x)\leq 1$ and $\phi_k''(x)\leq \frac{2}{kx}$ for $x>0$ this gives the uniform in $k$ upper bound
	\begin{align*}
		\mathbf 1_{\{u<-1\}}+\mathbf 1_{\{-1\leq u\leq 0\}}\ee^{-u} (\ee^u-1)^2
	\end{align*}
	which is integrable with respect to $P\otimes \mathbf 1_{\{s\leq t\}}\dd s\otimes \Pi(\dd u)$. Hence, (\ref{l1}) follows from the dominated convergence theorem.

Now we turn to prove \eqref{l2} which is more delicate as \eqref{l2} corresponds to the jumps possibly contradicting \eqref{mono} if $u$ was positive, i.e. if the L\'evy process $\xi$ was not spectrally negative. We only use Taylor's theorem and the mean-value theorem to estimate 	
	\begin{align*}
		&\quad \mathbf 1_{\{X_s>Y_s\}}\left(\frac{1}{Y_s}-\frac{1}{X_s}\right)\left(\phi_k(X_s-Y_s-Y_s(\ee^u-1))-\phi_k(X_s-Y_s)+\phi_k'(X_s-Y_s)Y_s(\ee^u-1)\right)\\
  	           &\leq \mathbf 1_{\{X_s>Y_s\}}\frac{1}{2}\phi_k''(\aleph_s)\frac{X_s-Y_s}{X_sY_s}Y_s^2(\ee^u-1)^2,
	\end{align*}
  	where  $\aleph_s\in (X_s-Y_s,X_s-Y_s+Y_s(1-\ee^u))$. The right-end of the interval is larger than the left-end due to the assumption $u\leq 0$. If $u$ was positive, then $X_s-Y_s+Y_s(1-\ee^u)$ might even become negative. In this case we should estimate $\phi_k''$ uniformly in $k$ in an interval containing zero but, by definition, $\phi_k''$ explodes as $k$ tends to infinity around zero.	
Using again that $\phi''_k(x)\leq \frac{2}{kx}$ for $x>0$,
 we obtain the $P\otimes \mathbf 1_{\{s\leq t\}}\dd s\otimes \Pi(\dd u)$-integrable upper bound
 $(\ee^u-1)^2$ so that (\ref{l2}) and thus (\ref{l3}) is verified.

To verify (\ref{l3b}), we proceed similarly. First note that for $x,y\geq 0$
 \begin{align*}
	&\quad \int_0^\infty  D_{{\bf h}(x,r)-{\bf h}( y,r)}	    \phi\big(x-y\big)\,\dd r\\
	&=\mathbf 1_{\{x>y\}}\int_0^{1/x}\big(|x-y+(-x+y)|-|x-y|-(-x+y)\big)\,\dd r\\
	&\quad+\mathbf 1_{\{x>y\}}\int_{1/x}^{1/y}\big(|x-y+y|-|x-y|-y\big)\,\dd r\\
	&\quad+\mathbf 1_{\{x\leq y\}}\int_0^{1/y}\big(|x-y+(-x+y)|-|x-y|+(-x+y)\big)\,\dd r\\
	&\quad+\mathbf 1_{\{x\leq y\}}\int_{1/y}^{1/x}\big(|x-y+(-x)|-|x-y|+(-x)\big)\,\dd r\\
	&=0
 \end{align*}
 since all integrands are constant zero. In particular, this shows that	
	\begin{align*}
		\E\left(q \int_0^{t\wedge \tau_m}\int_0^\infty
               D_{{\bf h}(X_{s},r)-{\bf h}( Y_{s},r)}	    \phi\big(X_{s}-Y_s\big)\,\dd s\,\dd r\right)=0
	\end{align*}
	so that (\ref{l3b}) follows if we can show that
	\begin{align*}%\label{bbb}\begin{split}
		&\quad \lim_{k\to \infty} \E\left( q\int_0^{t\wedge \tau_m}\int_0^\infty    D_{{\bf h}(X_{s},r)-{\bf h}( Y_{s},r)}	    \phi_k\big(X_{s}-Y_s\big)\,\dd s\,\dd r	 \right)\\
		&= \E\left( q\int_0^{t\wedge \tau_m}\int_0^\infty      D_{{\bf h}(X_{s},r)-{\bf h}( Y_{s},r)}	    \phi\big(X_{s}-Y_s\big)\,\dd s\,\dd r \right).%\end{split}
	\end{align*}

As before, separating the four cases and integrating out $r$ we see that we can equally show that
 for any $t\geq 0$ (the corresponding two cases for  $X_s\leq Y_s$ are similar)
	\begin{align}\label{l1b}\begin{split}
		&\lim_{k\to\infty} \E\left(q\int_0^t\mathbf 1_{\{X_s>Y_s\}}\frac{1}{X_s}\Big(
		-\phi_k(X_s-Y_s)+\phi_k'(X_s-Y_s)(X_s-Y_s)\Big){\dd s}\right)\\
		&= \E\left(q\int_0^t\mathbf 1_{\{X_s>Y_s\}}\frac{1}{X_s}\Big(
		-\big|X_s-Y_s\big|+X_s-Y_s\Big){\dd s}\right)\end{split}
	\end{align}
	and	
	\begin{align}\label{l2b}\begin{split}
		&\lim_{k\to\infty} \E\left(q\int_0^t\mathbf 1_{\{X_s>Y_s\}}\Big(\frac{1}{Y_s}-\frac{1}{X_s}\Big)\Big(\phi_k(X_s)-\phi_k(X_s-Y_s)-\phi_k'(X_s-Y_s)Y_s\Big){\dd s}\right)\\
		&= \E\left(q\int_0^t\mathbf 1_{\{X_s>Y_s\}}\left(\frac{1}{Y_s}-\frac{1}{X_s}\right)\Big(\big|X_s\big|-\big|X_s-Y_s\big|-Y_s\right)
           {\dd s}\Big).\end{split}
	\end{align}
Here we note that, as earlier, all the integrands are non-negative.
The convergence of (\ref{l1b}) follows by dominated convergence, since, using that $\phi_k(x)\geq 0$, $x\in\R$,
 and $\phi_k'(x)\leq 1$, $x\geq 0$, we have the $P\otimes \mathbf 1_{\{s\leq t\}}\dd s$-integrable
 upper bound
 \begin{align*}
   \mathbf 1_{\{X_s>Y_s\}}\frac{1}{X_s}\Big(-\phi_k(X_s-Y_s)+\phi_k'(X_s-Y_s)(X_s-Y_s)\Big)
      \leq \mathbf 1_{\{X_s>Y_s\}} \frac{X_s-Y_s}{X_s}\leq 1
 \end{align*}
 for the integrand.
For (\ref{l2b}) we use dominated convergence based on
	\begin{align*}
		 {\mathbf 1_{\{X_s>Y_s\}}}&\Big(\frac{1}{Y_s}-\frac{1}{X_s}\Big)
             \big(\phi_k((X_s-Y_s)+Y_s)-\phi_k(X_s-Y_s) -\phi_k'(X_s-Y_s)Y_s\Big) \\
            & = { \mathbf 1_{\{X_s>Y_s\}}} \Big(\frac{X_s-Y_s}{X_sY_s}\Big){\frac{1}{2}}\phi_k''(\aleph_s) Y_s^2
	\end{align*}
	for some $\aleph_s\in (X_s-Y_s,X_s-Y_s+Y_s)$, so that $\phi_k''(x)\leq \frac{2}{kx}$, $x>0$, gives the
 $P\otimes \mathbf 1_{\{s\leq t\}}\dd s$-integrable upper bound ${\frac{1}{k X_sY_s} Y_s^2\leq 1}$, $k\in\N$.
This completes the proof.
\end{proof}

We now turn our attention to the existence of solutions;
 the proof is given via a sequence of lemmas utilizing ideas of Fu and Li \cite{FL}.
The strategy is to construct solutions by first considering the compensated equation suppressing all jumps,
 then adding jumps carefully (with truncations that allow to do this) and finish with a weak convergence
 argument.

We start with proving existence of a unique strong solution for the {(ordinary) SDE obtained from \eqref{eqn}
 by truncations and omitting the jumps.
Namely, for all $z\geq 0$ and $0<\eps<m$, let us consider the SDE
 \begin{align}\label{E:aaa}
 \begin{split}
  Z_t&=z+\big({\log E\big(\ee^{\xi_1};\zeta>1\big)\big)t+ }\sigma \int_0^{t} \sqrt{(Z_s\wedge m)} \,\dd B_s\\
     &\quad +q{ \int_0^t\int_0^{1/\eps}   \mathbf 1_{\{r Z_s\leq 1\}}(Z_s\wedge m)\, \dd s\, \dd r}\\
     &\quad - \int_0^t\int_0^{1/\eps}\int_{|u|\geq \eps}
                 \mathbf 1_{\{r Z_s\leq 1\}}(Z_s\wedge m)(\ee^u-1)\,\dd s\,\dd r\,\Pi(\dd u).
 \end{split}
 \end{align}
An easy calculation shows that this is the same as the SDE
  \begin{align}\label{eqn:o}
  \begin{split}
	Z_t  &= z +  \big(\log E\big(\ee^{\xi_1};\zeta>1\big)\big)t
                  + \sigma \int_0^{t} \!\!\sqrt{Z_s\wedge m} \,\dd B_s\\
         &\phantom{=\,} - \left(\int_{|u|\geq \eps}(\ee^u-1)\,\Pi(\dd u)-q\right) \int_0^tb^{\eps,m}(Z_s)\,\dd s,
  \end{split}
 \end{align}	
 where
 \begin{align*}
  b^{\eps,m}(x)
     :=\mathbf 1_{\{x>0\}}\frac{x\wedge m}{x\vee \eps}
      =\begin{cases}
				\frac{m}{x}& \text{ if \ $x\geq m$,}\\
				1& \text{ if \ $\eps\leq x\leq m$,}\\
				\frac{x}{\eps} & \text{ if \ $0\leq x\leq \eps$,}\\
				0&\text{ if \ }x<0.
	   \end{cases}
 \end{align*}}

\begin{lemma}\label{L:aaa}
Suppose that $\log E\big(\ee^{\xi_1};\zeta>1\big)>0$ and $z\geq 0$.
Then, for all $0<\eps<m$, there is a pathwise unique non-negative strong solution
 to the (ordinary) SDE {\eqref{eqn:o} for $t\geq 0$.}
\end{lemma}

\begin{proof}
First let us consider the modified equation
 \begin{align}\label{eqn:ob}
  \begin{split}
	Z_t &= z + \big({\log E\big(\ee^{\xi_1};\zeta>1\big)\big)t
             - \left(\int_{|u|\geq \eps} (\ee^u-1)\,\Pi(\dd u)-q\right)} \int_0^tb^{\eps,m}(Z_s)\,\dd s\\
        & \quad+ \sigma \int_0^{t} \sqrt{(Z_s\wedge m)\vee 0} \,\dd B_s, \quad  t \geq 0,
  \end{split}
 \end{align}	
 so that the integrand of the Brownian integral is a priori well-defined.
The drift is globally Lipschitz (since it is piecewise continuously differentiable with bounded derivatives)
 and further $\vert \sqrt{x} -\sqrt{y}\vert\leq \sqrt{\vert x-y\vert}$, $x,y\geq 0$,
 so that the results of Yamada and Watanabe (see, e.g. Ikeda and Watanabe \cite[{Theorems IV.2.3 and IV.3.2}]{IW})
 ensure that the SDE \eqref{eqn:ob} has a pathwise unique strong solution for all non-negative
 initial values $z\geq 0$.
The non-negativity of the solution follows directly from the positivity of the drift close to zero.
To give a precise argument we utilize Proposition 2.1 of  Fu and Li \cite{FL}:
 Indeed, the functions
 \[
    \R\ni x\mapsto{\log E\big(\ee^{\xi_1};\zeta>1\big)
                                   - \left(\int_{|u|\geq \eps} (\ee^u-1)\,\Pi(\dd u)-q\right) \,b^{\eps,m}(x)}
                                 \]
                              and
                              \[
    \R\ni x\mapsto \sqrt{(x\wedge m) \vee 0}
 \]     	
 are continuous, $\sqrt{(x\wedge m) \vee 0}=0$ if $x\leq 0$ and
  \[
{    \log E\big(\ee^{\xi_1};\zeta>1\big)
        - \left(\int_{|u|\geq \eps} (\ee^u-1)\,\Pi(\dd u)-q\right)} \, b^{\eps,m}(x)
       >0, \qquad x\leq 0.
 \]
 Since solutions are non-negative we can drop the additional maximum with $0$ in the integrand
 for the Brownian integral so that there is a pathwise unique non-negative strong solution to (\ref{eqn:o}).
\end{proof}

Next, we add jumps to the SDE (\ref{eqn:o}) {(or equivalently to (\ref{E:aaa}))}.
By the choice of the truncation, jumps come with finite rate so that they can be added merely "by hands"
 (the procedure is called interlacing).

\begin{lemma}
Suppose that {$\log E\big(\ee^{\xi_1};\zeta>1\big)>0$} and $z\geq 0$.
Then, for all $0<\eps<m$, there is a pathwise unique non-negative strong solution to
	 \begin{align}\label{eqn:o2}
	\begin{split}
		Z_t	&=z+\big({\log E\big(\ee^{\xi_1};\zeta>1\big)\big)t}+ \sigma \int_0^{t} \sqrt{(Z_s\wedge m)} \dd B_s\\
		&\quad{-\int_0^t\int_0^{1/\eps} \mathbf 1_{\{r Z_{s-}\leq 1\}} (Z_s\wedge m) (\mathcal M-\mathcal M')(\dd s,\dd r)}\\
            &\quad +\int_0^t\int_0^{1/\eps}\int_{|u|\geq \eps}   \mathbf 1_{\{r Z_{s-}\leq 1\}}
                (Z_{s-}\wedge m)(\ee^u-1) (\mathcal N-\mathcal N')(\dd s,\dd r,\dd u)
	\end{split}
	\end{align}
 for $t\geq 0$.
\end{lemma}

\begin{proof}
{First note that the} existence of a pathwise unique non-negative strong solution to the SDE (\ref{E:aaa}) {for $t\geq 0$
 follows by} Lemma \ref{L:aaa}.
Next, we add the jumps that are chosen according to  independent Poisson random measures $\mathcal N$ on
 $(0,\infty)\times (0,\infty)\times \R$ with intensity measure
  $\mathcal N'(\dd s,\dd r,\dd u)=\dd s\,\otimes\,\dd r\,\otimes\,\Pi(\dd u)$ {and  $\mathcal M$ on
 $(0,\infty)\times (0,\infty)$ with intensity measure
  $\mathcal M'(\dd s,\dd r)={q} \dd s\,\otimes\,\dd r$}.
Note that the truncated SDE (\ref{eqn:o}) was chosen in a way that it only contains the compensation on sets with finite intensity.
Hence, only finitely many jumps have to be added in finite time intervals.
Those can be added via interlacing, i.e. using the pathwise unique strong {solution of} the SDE (\ref{eqn:o}),
 {as in the proof of Proposition 2.2 in Fu and Li \cite{FL}.}
{To keep the proof shorter we assume $q=0$ since the additional integral corresponding to killing
 can be dealt with likewise.}
Let $\{S_k : k\in\N\}$ be the set of jump times of the Poisson process
 \[
   \R\ni t\mapsto \int_0^t \int_0^{1/\varepsilon}\int_{\vert u\vert\geq \varepsilon}
                         1\,\cN(\dd s,\dd r,\dd u)
                         = \cN\big((0,t)\times (0,1/\eps)\times\{u\in\R : \vert u\vert\geq \eps\}\big).
 \]
Clearly, $S_k$ is a sum of independent exponentially distributed random variables
 with parameter $\frac{1}{\eps}\Pi(\vert u\vert\geq \eps)$ and hence we have $S_k\to\infty$ almost surely
 as $k\to\infty$.
For $0\leq t<S_1$, let $ Z^{\eps,m}_t$ be the solution {of \eqref{E:aaa} with $q=0$ given by} Lemma \ref{L:aaa}.
To set up an induction, suppose that $ Z^{\eps,m}_t$ has been defined for $0\leq t<S_k$, and let
 \begin{align}\label{help15}
   \xi:= Z^{\eps,m}_{S_k-} + \int_{\{S_k\}}\int_0^{1/\eps}\int_{|u|\geq \eps}
                                 \mathbf 1_{\{r Z^{\eps,m}_{s-}\leq 1\}}
                                ( Z^{\eps,m}_{s-}\wedge m)(\ee^u-1) \mathcal N(\dd s,\dd r,\dd u).
 \end{align}
Note that $\xi$ is non-negative which follows by the simple fact that
 {$x+\mathbf 1_{\{x\leq \eps\}}(x\wedge m)(\ee^u-1)\geq 0$,} $x\geq 0$, $u\in\R$.
Since the SDE (\ref{eqn:o}) has a pathwise unique non-negative strong solution, there is also a pathwise
 unique non-negative strong solution, say $(X_k(t))_{t\geq 0}$, to the SDE
 \begin{align}\label{help16}
 \begin{split}
     Z_t&=\xi+(\log E\big(\ee^{\xi_1}\big))t+ \sigma \int_0^{t} \sqrt{(Z_s\wedge m)} \,\dd B_{S_k+s}\\
             &\quad- \int_{|u|\geq \varepsilon}(\ee^u-1) \,\Pi(\dd u)
                 \int_0^t b^{\eps,m}(Z_s)\,\dd s,\quad t\geq 0.
  \end{split}
 \end{align}
Here we call the attention to the fact that the SDE (\ref{eqn:o}) (and hence the SDE \eqref{help16}, too)
 has a pathwise unique non-negative strong solution for all non-negative
 {random} initial condition (indeed, Yamada and Watanabe's theorems and Proposition 2.1 in Fu and Li \cite{FL}
 are valid for non-negative {random} initial conditions, too).
We also note that the strong Markov property of $B$ and the independence of $B$ and $\cN$ yield
 that $(B_{S_k+t} - B_{S_k})_{t\geq 0}$ is a standard Brownian motion with respect to
 its natural filtration and it is independent of $\cG_{S_k+}$, see, e.g. Karatzas and Shreve
 \cite[Theorem 2.6.16]{KS}.
For $S_k\leq t < S_{k+1}$, we set $ Z_t^{\eps,m}:=X_k(t-S_k)$.
Then, for $S_k\leq t<S_{k+1}$, we have
 \begin{align*}
        Z^{\eps,m}_t
        &=X_k(t-S_k)
         = \xi + (\log E\big(\ee^{\xi_1}\big))(t-S_k)
              + \sigma \int_0^{t-S_k} \sqrt{(X_k(s)\wedge m)} \,\dd B_{S_k+s}\\
             &\quad- \int_{|u|\geq \varepsilon}(\ee^u-1) \,\Pi(\dd u)
                 \int_0^{t-S_k} b^{\eps,m}(X_k(s))\,\dd s\\
         &= Z^{\eps,m}_{S_k-}
            + \int_{\{S_k\}}\int_0^{1/\eps}\int_{|u|\geq \eps}
                                 \mathbf 1_{\{r Z^{\eps,m}_{s-}\leq 1\}}
                                ( Z^{\eps,m}_{s-}\wedge m)(\ee^u-1) \mathcal N(\dd s,\dd r,\dd u)\\
         &\quad + (\log E\big(\ee^{\xi_1}\big))(t-S_k)
                + \sigma \int_{S_k}^t \sqrt{(Z^{\eps,m}_s\wedge m)} \,\dd B_s\\
         &\quad- \int_{|u|\geq \varepsilon}(\ee^u-1) \,\Pi(\dd u)
                 \int_{S_k}^t b^{\eps,m}(Z^{\eps,m}_s)\,\dd s.
 \end{align*}
By the induction hypothesis, we obtain
 \begin{align*}
   Z^{\eps,m}_{S_k-}
      &= z + (\log E\big(\ee^{\xi_1}\big))S_k
                + \sigma \int_0^{S_k} \sqrt{(Z^{\eps,m}_s\wedge m)} \,\dd B_s\\
      &\phantom{=\;}  - \int_0^{S_k}\int_0^{1/\eps}\int_{|u|\geq \eps}
                                 \mathbf 1_{\{r Z^{\eps,m}_s\leq 1\}}
                                ( Z^{\eps,m}_s\wedge m)(\ee^u-1)
                                \,\dd s\,\dd r\,\Pi(\dd u)\\
      &\phantom{=\;} + \sum_{\ell=1}^{k-1} \int_{\{S_\ell\}}\int_0^{1/\eps}\int_{|u|\geq \eps}
                                 \mathbf 1_{\{r Z^{\eps,m}_{s-}\leq 1\}}
                                ( Z^{\eps,m}_{s-}\wedge m)(\ee^u-1) \mathcal N(\dd s,\dd r,\dd u),
  \end{align*}
so that, for $S_k\leq t<{S_{k+1}}$, $Z^{\eps,m}_t$ satisfies the SDE
 \begin{align}\label{help_SDE_interlacing}
   \begin{split}
	Z_t&=z+(\log E\big(\ee^{\xi_1}\big))t
                  + \sigma \int_0^{t} \sqrt{(Z_s\wedge m)} \dd B_s
       -\int_{|u|\geq \eps} (\ee^u-1)\,\Pi(\dd u)\int_0^{t} b^{\eps,m}(Z_s) \,\dd s\\
	   &\quad+\int_0^t\int_0^{1/\eps}\int_{|u|\geq \eps}\mathbf 1_{\{rZ_{s-}\leq 1\}}
                       (Z_{s-}\wedge m)(\ee^u-1) \mathcal N(\dd s,\dd r,\dd u).
  \end{split}
 \end{align}
Hence, by induction, this defines a process $( Z_t^{\eps,m})_{t\geq 0}$ which is a non-negative
 strong solution to the SDE \eqref{help_SDE_interlacing} for all $t\geq 0$.
By calculating the compensated integral in the SDE \eqref{eqn:o2} we get that
 the SDE \eqref{help_SDE_interlacing} is the same as the SDE \eqref{eqn:o2} yielding that
 $Z^{\eps,m}$ is a non-negative strong solution of \eqref{eqn:o2}.
The non-negativity of $Z^{\eps,m}$ follows by the construction.
Finally, pathwise uniqueness for the SDE \eqref{eqn:o2} follows from that of \eqref{eqn:o} {and \eqref{help16}}.
Indeed, since jumps come with finite rate, uniqueness holds between the jumps so that also the jumps themselves
 are uniquely determined.
\end{proof}

So far, we got around the problem of having infinitely many small jumps by cutting the L\'evy measure $\Pi$ at
 $\eps$.
Note also that infinitely many large jumps were avoided by cutting the state-dependent jump intensity by $1/\eps$.
This allowed us to construct solutions by standard SDE theory and "by hands" (interlacing procedure).
Next, we get rid of these two restrictions via martingale problem arguments.
The additional truncation by $m$ remains in order to ensure tightness when $\eps$ goes to zero.

\begin{lemma}\label{lem:tightness}
Suppose that {$\log E\big(\ee^{\xi_1};\zeta>0\big)>0$}, $z\geq 0$ and for all $0<\eps<m$,
 let $Z^{\eps,m}$ be the pathwise unique non-negative strong solution to (\ref{eqn:o2}).
Then, {for every $m>0$ and every sequence $(\eps_n)_{n\in\N}$ tending to zero,}
 the family {$\{ Z^{\eps_n,m}:n\in\N\}$} is tight in Skorokhod's $J_1$ topology on $\D$.
\end{lemma}

\begin{proof}
For the proof we apply Aldous's tightness criterion (see Aldous \cite{Aldous2} or Chapter 3.8 of Ethier and Kurtz \cite{EK}).
According to this, to prove that  {$\{ Z^{\eps_n,m}:n\in\N\}$} is tight in $\D$ it is enough to show that
 \begin{itemize}
   \item[(i)] for every fixed $t\geq 0$, the set of random variables  {$\{ Z^{\eps_n,m}_t:n\in\N\}$}
              is tight,
   \item[(ii)] for every sequence of stopping times $(\tau_n)_{n\in\N}$ (with respect to the
               filtration $(\cG_t)_{t\geq 0}$) bounded above by $T>0$ and for every sequence of
               positive real numbers $(\delta_n)_{n\in\N}$ converging to $0$,
               {$Z^{\eps_n,m}_{\tau_n +\delta_n} -  Z^{\eps_n,m}_{\tau_n} \to 0$
               in probability as $n\to\infty$.}
 \end{itemize}
To prove (i), by Markov's inequality it is enough to check that, for every fixed $t\geq 0$,
 \begin{align}\label{help17}
  \sup_{n\in\N} E\Big[  (Z^{{\eps_n},m}_t)^2 \Big]<\infty.
 \end{align}	
Using that $(a+b+c+d{+e})^2\leq {5(a^2+b^2+c^2+d^2+e^2)}$, ${a,b,c,d,e\in\R}$, we obtain that
 $E\Big[  ( Z^{{ \eps_n},m}_t)^2 \Big]$ can be bounded by
 \begin{align*}
  5z^2 & + {5\big(\log E\big(\ee^{\xi_1};\zeta>1\big)\big)^2 t^2}
          + 5\sigma^2 E\left(
              \left(\int_0^t  \sqrt{\big( Z^{{ \eps_n},m}_s\wedge m\big)} \,\dd B_s \right)^2 \right)\\
                     & {+ 5 E\left(
                \left(
                  \int_0^t\int_0^{1/{ \eps_n}}
                        \mathbf 1_{\{rZ^{{ \eps_n},m}_{s-}\leq 1 \}}
                  ( Z^{{ \eps_n},m}_{s-}\wedge m)(\mathcal M-\mathcal M')(\dd s,\dd r)
                \right)^2
               \right)}\\
       & + 5 E\left(
                \left(
                  \int_0^t\int_0^{1/{ \eps_n}}\int_{|u|\geq { \eps_n}}
                        \mathbf 1_{\{rZ^{{ \eps_n},m}_{s-}\leq 1 \}}
                  ( Z^{{ \eps_n},m}_{s-}\wedge m)(\ee^u-1) (\mathcal N-\mathcal N')(\dd s,\dd r,\dd u)
                \right)^2
               \right),
              \end{align*}
 which, by It\=o's isometry and page 62 in Ikeda and Watanabe \cite{IW},
 can be bounded from above by
               \begin{align*}
     & \quad 5z^2 + {5\big(\log E\big(\ee^{\xi_1};\zeta>1\big)\big)^2 t^2}
             + 5\sigma^2 E\left(\int_0^t ( Z^{{ \eps_n},m}_s\wedge m)\,\dd s \right)\\
     &\quad {+ 5E\left(q
                  \int_0^t\int_0^{1/{ \eps_n}}
                        \mathbf 1_{\{rZ^{{ \eps_n},m}_s\leq 1 \}}
                   ( Z^{{ \eps_n},m}_s\wedge m)^2\,\dd s\,\dd r
                 \right)}\\
     &\quad  + 5E\left(
                  \int_0^t\int_0^{1/{ \eps_n}}\int_{|u| \geq { \eps_n} }
                        \mathbf 1_{\{rZ^{{ \eps_n},m}_s\leq 1 \}}
                   ( Z^{{ \eps_n},m}_s\wedge m)^2(\ee^u-1)^2\,\dd s\,\dd r\,\Pi(\dd u)
                 \right)\\
     & \leq 5z^2 + {5\big(\log E\big(\ee^{\xi_1};\zeta>1\big)}\big)^2 t^2 + 5\sigma^2 mt
                 +{5E\left(q\int_0^t
                 (  Z^{{ \eps_n},m}_s\wedge m)\,\dd s
                 \right)}\\
    &\quad        + 5E\left(\int_0^t\int_{|u|\geq { \eps_n}}
                   ( Z^{{ \eps_n},m}_s\wedge m)(\ee^u-1)^2\,\dd s\,\Pi(\dd u)
                 \right)\\
     & \leq 5z^2 + {5\big(\log E\big(\ee^{\xi_1};\zeta>1\big)\big)^2} t^2
            + 5mt\left(\sigma^2 +q+ \int_{|u|\geq { \eps_n}} (\ee^u-1)^2\,\Pi(\dd u) \right),
            \quad t\geq 0,
 \end{align*}
 which is uniformly bounded in ${ n\in\N}$ since $E(\ee^{2\xi_1})<\infty$ and $\Pi$ is a L\'evy measure.

Now we turn to (ii) proving the stronger statement that $Z^{{ \eps_n},m}_{\tau_n+\delta_n} -  Z^{{ \eps_n},m}_{\tau_n}$
 converges to 0 in $L^2$ as ${ n\to\infty}$.
Namely, by the SDE \eqref{eqn:o2} and using that {$(a+b+c+d)^2\leq 4(a^2+b^2+c^2+d^2)$}, we have
 \begin{align*}
   &E\left( \vert  Z^{{ \eps_n},m}_{\tau_n+\delta_n} -  Z^{{ \eps_n},m}_{\tau_n}\vert^2 \right)\\
   &= E\Bigg[\Bigg\vert \big({\log E\big(\ee^{\xi_1};\zeta>1\big)\big)\delta_n}
               + \sigma \int_{\tau_n}^{\tau_n+\delta_n}  \sqrt{\big( Z^{{ \eps_n},m}_s\wedge m\big)} \,\dd B_s\\
               &\quad\quad{-\int_{\tau_n}^{\tau_n+\delta_n}\int_0^{1/{ \eps_n}}
                        \mathbf 1_{\{rZ^{{ \eps_n},m}_{s-}\leq 1\}}
                  ( Z^{{ \eps_n},m}_{s-}\wedge m) (\mathcal M-\mathcal M')(\dd s,\dd r)}\\
   &\quad \quad   + \int_{\tau_n}^{\tau_n+\delta_n}\int_0^{1/{ \eps_n}}\int_{|u|\geq { \eps_n}}
                        \mathbf 1_{\{rZ^{{ \eps_n},m}_{s-}\leq 1\}}
                  ( Z^{{ \eps_n},m}_{s-}\wedge m)(\ee^u-1) (\mathcal N-\mathcal N')(\dd s,\dd r,\dd u)
           \Bigg\vert^2 \Bigg]
  \end{align*}
  \begin{align*}
   &\leq 4{\big(\log E\big(\ee^{\xi_1};\zeta>1\big)\big)^2\delta_n^2}
        + 4\sigma^2 E\left(\int_{\tau_n}^{\tau_n+\delta_n}\sqrt{\big( Z^{{ \eps_n},m}_s\wedge m\big)}
                            \,\dd B_s\right)^2\\
  &\quad {+ 4E\left(\int_{\tau_n}^{\tau_n+\delta_n}\int_0^{1/{ \eps_n}}
                     \mathbf 1_{\{rZ^{{ \eps_n},m}_{s-}\leq 1\}}
                     ( Z^{{ \eps_n},m}_{s-}\wedge m)
                     (\mathcal M-\mathcal M')(\dd s,\dd r) \right)^2}\\
  &\quad + 4E\left(\int_{\tau_n}^{\tau_n+\delta_n}\int_0^{1/{ \eps_n}}\int_{|u|\geq { \eps_n}}
                     \mathbf 1_{\{rZ^{{ \eps_n},m}_{s-}\leq 1\}}
                     ( Z^{{ \eps_n},m}_{s-}\wedge m)(\ee^u-1)
                     (\mathcal N-\mathcal N')(\dd s,\dd r,\dd u) \right)^2.
 \end{align*}
By Proposition 3.2.10 in Karatzas and Shreve \cite{KS},
 \begin{align*}
    E\left(\int_{\tau_n}^{\tau_n+\delta_n}\sqrt{\big( Z^{{ \eps_n},m}_s\wedge m\big)}
                            \,\dd B_s\right)^2
      = E\left(\int_{\tau_n}^{\tau_n+\delta_n} ( Z^{{ \eps_n},m}_s\wedge m) \,\dd s \right),
 \end{align*}
 and, by Theorem II.1.33 in Jacod and Shiryaev \cite{JS}, the optimal stopping theorem, using
 the same arguments as in the proof of (3.2.22) in Karatzas and Shreve \cite{KS}, we get{
 \begin{align*}
   &E\left(\int_{\tau_n}^{\tau_n+\delta_n}\int_0^{1/{ \eps_n}}
                     \mathbf 1_{\{rZ^{{ \eps_n},m}_{s-}\leq 1\}}
                     ( Z^{{ \eps_n},m}_{s-}\wedge m)
                     (\mathcal M-\mathcal M')(\dd s,\dd r) \right)^2\\
   &\quad= E\left(q  \int_{\tau_n}^{\tau_n+\delta_n}\int_0^{1/{ \eps_n}}
                \mathbf 1_{\{r Z^{{ \eps_n},m}_s\leq 1\}}( Z^{{ \eps_n},m}_s\wedge m)^2 \,\dd s\,\dd r\ \right)
 \end{align*}
 and}
 \begin{align*}
   &E\left(\int_{\tau_n}^{\tau_n+\delta_n}\int_0^{1/{ \eps_n}}\int_{|u|\geq { \eps_n}}
                     \mathbf 1_{\{rZ^{{ \eps_n},m}_{s-}\leq 1\}}
                     ( Z^{{ \eps_n},m}_{s-}\wedge m)(\ee^u-1)
                     (\mathcal N-\mathcal N')(\dd s,\dd r,\dd u) \right)^2\\
   &\quad= E\left(  \int_{\tau_n}^{\tau_n+\delta_n}\int_0^{1/{ \eps_n}}\int_{|u|\geq { \eps_n}}
                     \mathbf 1_{\{r Z^{{ \eps_n},m}_s\leq 1\}}
                     ( Z^{{ \eps_n},m}_s\wedge m)^2(\ee^u-1)^2 \,\dd s\,\dd r\,\Pi(\dd u) \right).
 \end{align*}
Hence,
 \begin{align*}
  E\left( \vert  Z^{{ \eps_n},m}_{\tau_n+\delta_n} -  Z^{{ \eps_n},m}_{\tau_n}\vert^2 \right)
   \leq 4\big(\log E\big(\ee^{\xi_1};\zeta>1\big)\big)^2\delta_n^2
    + 4m\delta_n\left( \sigma^2 + q+\int_{|u|\geq { \eps_n}} (\ee^u-1)^2\,\Pi(\dd u) \right),
 \end{align*}
  which tends to zero as $n$ tends to infinity. This proves (ii) and finishes the proof of the lemma.
\end{proof}

Having proved the tightness we identify the limit points as weak solutions to

 \begin{align}\label{eqn:o3}
 \begin{split}
    Z_t&=z+\big({\log E\big(\ee^{\xi_1};\zeta>1\big)}\big)t+ \sigma \int_0^{t} \sqrt{(Z_s\wedge m)} \dd B_s\\
       &\phantom{=\,}    -\int_0^t\int_0^\infty (Z_{s-}\wedge m)(\mathcal M-\mathcal M')(\dd s,\dd r)\\
	   &\phantom{=\,}     +\int_0^t\int_0^\infty\int_{\R}\mathbf 1_{\{rZ_{s-}\leq 1\}}
                             (Z_{s-}\wedge m)(\ee^u-1) (\mathcal N-\mathcal N')(\dd s,\dd r,\dd u).
  \end{split}
 \end{align}
This is achieved via martingale problem characterizations of jump type SDEs.

\begin{lemma}\label{lem:weakexistence}
{Let $m>0$ and $Z^m$ be} any limiting point of the tight family
 ${\{Z^{\eps_k,m} : k\in\N\}}$ {(given in Lemma \ref{lem:tightness})} as ${\eps_k\downarrow 0}$.
Then $Z^m$ is a weak solution to (\ref{eqn:o3}) {for $t\geq 0$}.
 \end{lemma}

\begin{proof}
First note that, due to It\=o's formula, $ Z^{{\eps_k},m}$ solves the following martingale problem:
 \begin{align*}
	M^{{\eps_k},m}_t:=f( Z^{{\eps_k},m}_t)-f(z)-\int_0^t (\cA^{{\eps_k},m} f)( Z^{{\eps_k},m}_s)\,\dd s,
                   \quad t\geq 0,
 \end{align*}
 is a martingale for all $f\in C_b^2(\R_+)$, where
 \begin{align*}
  (\cA^{{\eps_k},m}f)(x)&:=\frac{1}{2}\sigma^2(x\wedge m)f''(x)+\big(\log E\big(\ee^{\xi_1};\zeta>1\big)\big)f'(x)\\
      	   	        & \quad {+ q\int_0^{1/{\eps_k}}
                           \Big(f\big(x-\mathbf 1_{\{rx\leq 1\}}(x\wedge m)\big)-f(x)
                            +f'(x)\mathbf 1_{\{rx\leq 1\}}(x\wedge m)\Big)
                                 \,\dd r}\\
      	   	        & \quad + \int_0^{1/{\eps_k}} \int_{|u|\geq {\eps_k}}
                           \Big(f\big(x+\mathbf 1_{\{rx\leq 1\}}(x\wedge m)(\ee^u-1)\big)-f(x)\\
   &\phantom{\quad + \int_0^{1/{\eps_k}} \int_{|u|\geq {\eps_k}} \Big(}
                            -f'(x)\mathbf 1_{\{rx\leq 1\}}(x\wedge m)(\ee^u-1)\Big)
                                 \,\dd r\,\Pi(\dd u),\quad x\geq 0,
 \end{align*}
 and $C_b^2(\R_+)$ denotes the set of bounded twice continuously differentiable functions
 $f:\R_+\to\R$ of which the first and second derivatives are also bounded.
{Without loss of generality we can assume that $(\eps_k)_{k\in\N}$} realizes the weak convergence of
 $Z^{\eps_k,m}$ to $Z^m$ in $\D$ as $k\to\infty$.
By Skorokhod's representation theorem and Riesz's theorem we may assume that $ Z^{\eps_k,m}$
 converges to $Z^{m}$ almost surely in $\D$ (possibly on a different probability space and
 changing also the subsequence $(\eps_k)_{k\in\N}$).
By Ethier and Kurtz \cite[page 118]{EK}, this yields that $P(\bar\Omega)=1$, where
 \[
    \bar\Omega:=\left\{\omega\in\Omega : \lim_{k\to\infty} Z^{\eps_k,m}_t(\omega) = Z^m_t(\omega)
                             \quad \text{for $t\geq 0$ at which $(Z^m_u(\omega))_{u\geq 0}$
                                         is continuous}\right\}.
 \]

Next we show that for all $f\in C_b^2(\R_+)$,
 \begin{align}\label{mm}
		M^{m}_t&:=f( Z^{m}_t)-f(z)-\int_0^t (\cA^{m} f)( Z^{m}_s)\,\dd s,\quad t\geq 0,
 \end{align}
 is a martingale with respect to its natural filtration,
 where
 \begin{align*}
  (\cA^{m}f)(x)&:=\frac{1}{2}\sigma^2(x\wedge m)f''(x)+\big(\log E\big(\ee^{\xi_1};\zeta>1\big)\big)f'(x)\\
		     &\quad+q{ \int_0^\infty\Big(f(x-\mathbf 1_{\{rx\leq 1\}}
                    (x\wedge m))-f(x)
                 +f'(x)\mathbf 1_{\{rx\leq 1\}}(x\wedge m)\Big)
                     \,\dd r}\\
		     &\quad+\int_0^\infty\int_{\R}\Big(f(x+\mathbf 1_{\{rx\leq 1\}}
                    (x\wedge m)(\ee^u-1))-f(x)\\
             &\phantom{\quad+\int_0^\infty\int_{\R}\Big(}
                 -f'(x)\mathbf 1_{\{rx\leq 1\}}(x\wedge m)(\ee^u-1)\Big)
                     \,\dd r\,\Pi(\dd u),
                     \quad x\geq 0.
 \end{align*}		
In what follows let $\omega\in\bar\Omega$ be fixed.
First, using the dominated convergence theorem, we show that
 \begin{align}\label{help18}
   \lim_{k\to\infty} \int_0^t (\cA^{\eps_k,m}f)( Z^{\eps_k,m}_s)(\omega)\,\dd s
        = \int_0^t (\cA^{m}f)( Z^{m}_s)(\omega)\,\dd s,
        \qquad t\geq 0.
 \end{align}
Let us introduce the notation
 \[
   {D_m}(\omega):=\big\{ t\geq 0 : \text{$(Z^m_u(\omega))_{u\geq 0}$ is continuous at $t$} \big\},
     \quad {m>0.}
 \]
Then $\lim_{k\to\infty} Z^{\eps_k,m}_t(\omega) = Z^m_t(\omega)$ for all $t\in {D_m}(\omega)$, and,
 by Ethier and Kurtz \cite[page 114]{EK}, $[0,\infty)\setminus {D_m}(\omega)$ is at most countable.
Using the boundedness and continuity of $f$, $f'$ and $f''$, we check that for the integrands
 of the left-hand side of \eqref{help18}
 \begin{align}\label{he}
     \lim_{k\to\infty}(\cA^{\eps_k,m}f)( Z^{\eps_k,m}_s )(\omega)=(\cA^{m}f)( Z^m_s)(\omega)
      \qquad
       \text{for all $s\in {D_m}(\omega)$.}
 \end{align}
Since $s\in {D_m}(\omega)$ and $f'$, $f''$ are continuous, the first {two} summands in
 $(\mathcal A^{\eps_k,m}f)(Z^{\eps_k,m}_s)(\omega)$ trivially converge as $k\to\infty$
 {to that of $(\mathcal A^mf)(Z^m_s)(\omega)$}
 so that it suffices to show the convergence of the integrals in
 $(\mathcal A^{\eps_k,m}f)(Z^{\eps_k,m}_s)(\omega)$. {In the following we show the convergence for the second integral, the arguments for the integral corresponding to killing are similar (replacing $(\ee^u-1)^2$ by $1$ and skipping
 the integration with respect to $u$).}\\
On the one hand, for all ${m>0}$ and $s\in {D_m}(\omega)$, the integrand (as a function of $r$ and $u$)
 \begin{align*}
  \mathbf 1_{\{r\leq 1/\eps_k\}} \mathbf 1_{\{\vert u\vert\geq \eps_k\}}
                   & f\big(Z^{\eps_k,m}_s(\omega)+\mathbf 1_{\{rZ^{\eps_k,m}_s(\omega)\leq 1\}}
                    (Z^{\eps_k,m}_s(\omega)\wedge m)(\ee^u-1)\big)-f(Z^{\eps_k,m}_s(\omega))\\
                   & - f'(Z^{\eps_k,m}_s(\omega))\mathbf 1_{\{rZ^{\eps_k,m}_s(\omega)\leq 1\}}
                         (Z^{\eps_k,m}_s(\omega)\wedge m)(\ee^u-1)
 \end{align*}
 converges, as $k\to\infty$, pointwise to
 \begin{align*}
   &f(Z^m_s(\omega)+\mathbf 1_{\{rZ^m_s(\omega)\leq 1\}}
      (Z^m_s(\omega)\wedge m)(\ee^u-1))-f(Z^m_s(\omega))\\
   &\qquad\qquad  - f'(Z^m_s(\omega))\mathbf 1_{\{rZ^m_s(\omega)\leq 1\}}(Z^m_s(\omega)\wedge m)(\ee^u-1)
 \end{align*}
 possibly except at the point $r=Z^m_s(\omega)$ (it can only occur if $Z^m_s(\omega)>0$).
On the other hand, by Taylor expansion of second order, we can derive the upper bound
\begin{align*}
 \begin{split}
  &\sup_{k\in\N}
  \mathbf 1_{\{r\leq 1/\eps_k\}}\mathbf 1_{\{\vert u\vert\geq \eps_k\}}
       \Big\vert f( Z^{\eps_k,m}_s(\omega)+\mathbf 1_{\{r Z^{\eps_k,m}_s(\omega)\leq 1\}}( Z^{\eps_k,m}_s(\omega)\wedge m)(\ee^u-1))
                             -f( Z^{\eps_k,m}_s(\omega))\\
  & \phantom{\sup_{k\in\N}\mathbf 1_{\{r\leq 1/\eps_k\}}\mathbf 1_{\{\vert u\vert\geq \eps_k\}}}
    - f'( Z^{\eps_k,m}_s(\omega))\mathbf 1_{\{r Z^{\eps_k,m}_s(\omega)\leq 1\}}
          ( Z^{\eps_k,m}_s(\omega)\wedge m)(\ee^u-1)
    \Big\vert\\
  &\quad \leq \frac{1}{2}\sup_{x\in\R_+}\vert f''(x)\vert
         \sup_{k\in\N} \mathbf 1_{\{r Z^{\eps_k,m}_s(\omega)\leq 1\}}
           ( Z^{\eps_k,m}_s(\omega)\wedge m)^2 (\ee^u-1)^2.
\end{split}
 \end{align*}
The upper bound is integrable with respect to $\dd r\otimes \Pi(\dd u)$, since an easy calculation shows that,
 for all ${m>0}$ and $s\in {D_m}(\omega)$,
 \[
     \mathbf 1_{\{r Z^{\eps_k,m}_s(\omega)\leq 1\}}( Z^{\eps_k,m}_s(\omega)\wedge m)^2
       \leq \left(\frac{1}{r}\wedge m\right)^2,
      \qquad k\in\N,
 \]
 and hence the integral in $(\mathcal A^{\eps_k,m}f)(Z^{\eps_k,m}_s)(\omega)$ is bounded above by
  \begin{align*}%\label{up}
   \begin{split}
    \frac{1}{2}\sup_{x\in\R_+}&\vert f''(x)| \int_0^\infty \left(\frac{1}{r}\wedge m\right)^2 \,\dd r
                \int_\R(\ee^u-1)^2\,\Pi(\dd u)\\
     & {=} \frac{1}{2}\sup_{x\in\R_+}\vert f''(x)|(2m)\int_\R(\ee^u-1)^2\,\Pi(\dd u).
  \end{split}
 \end{align*}
Using dominated convergence theorem we have (\ref{he}).
Next, using again dominated convergence theorem, we check \eqref{help18}.
By \eqref{he}, using also that $[0,\infty)\setminus {D_m}(\omega)$ is at most countable,
 for all $t\geq 0$,
 \[
    \lim_{k\to\infty}(\cA^{\eps_k,m}f)( Z^{\eps_k,m}_s )(\omega)=(\cA^{m}f)( Z^m_s)(\omega)
      \qquad
       \text{Lebesgue a.e. $s\in [0,t]$.}
 \]
Further, for all $s\in[0,t]$,
 \begin{align}\label{ka}\begin{split}
   \sup_{k\in\N} (\cA^{\eps_k,m}f)( Z^{\eps_k,m}_s )(\omega)
      &\leq \frac{{\sigma^2}m}{2}\sup_{x\in\R_+}\vert f''(x)\vert
           + \big(\log E(\ee^{\xi_1};\zeta>1)\big)\sup_{x\in\R_+}\vert f'(x)\vert\\
      &\quad +qm \sup_{x\in\R_+}\vert f''(x)\vert
              + m \sup_{x\in\R_+}\vert f''(x)\vert \int_\R(\ee^u-1)^2\,\Pi(\dd u),\end{split}
 \end{align}
 which is trivially integrable over $[0,t]$ with respect to $s$.
Hence \eqref{help18} follows by dominated convergence theorem.
Note also that, since $P(\bar\Omega)=1$ and almost sure convergence implies weak convergence,
 we have
 \begin{align}\label{help24}
   \left( \int_0^t (\cA^{\eps_k,m}f)( Z^{\eps_k,m}_s )\,\dd s\right)_{t\geq 0 }
    \quad \text{converges weakly to} \quad
   \left( \int_0^t (\cA^{m}f)( Z^{m}_s )\,\dd s\right)_{t\geq 0 }
 \end{align}
 in $\D$ as $k\to\infty$.

The martingale property of $M^m$ with respect to its own filtration can be checked as follows.
First, by Jacod and Shiryaev \cite[Corollary IX.1.19]{JS}, we show that $(M^m_t)_{t\geq 0}$ is a local martingale.
To check the conditions of this corollary we have to show that $M^{\eps_k,m}$ converges weakly
 in $\D$ as $k\to\infty$ to $M^m$ and that there is some $b\geq 0$ such that
 $\vert M^{\eps_k,m}_t - M^{\eps_k,m}_{t-}\vert\leq b$ for all $t>0$, $k,m\in\N$, almost surely.
Using that $Z^{\eps_k,m}$ converges weakly in $\D$ to $Z^m$ as $k\to\infty$ and $f$ is continuous
 and bounded, we have $f(Z^{\eps_k,m})$ converges weakly in $\D$ to $f(Z^m)$ as $k\to\infty$.
Since the limit process in \eqref{help24} is continuous, by Jacod and Shiryaev
 \cite[Proposition VI.1.23]{JS}, we obtain that $M^{\eps_k,m}$ converges weakly in $\D$ as $k\to\infty$
 to $M^m$.
Further, almost surely for all $t\geq 0$,
 \begin{align*}
   \vert M^{\eps_k,m}_t - M^{\eps_k,m}_{t-}\vert
      = \vert f(Z^{\eps_k,m}_t) - f(Z^{\eps_k,m}_{t-})\vert
      \leq 2 \sup_{x\in\R_+}\vert f(x)\vert<\infty.
 \end{align*}
The martingale property of the local martingale $M^m$ follows by Jacod and Shiryaev \cite[Proposition I.1.47]{JS},
 since, by the very same arguments as before, for all $T>0$,
 \begin{align*}
  E\left(\sup_{t\in[0,T]}\vert M_t^m\vert\right)
      &\leq 2\sup_{x\in\R_+}\vert f(x)\vert
           + T\Bigg(\frac{{\sigma^2}m}{2}\sup_{x\in\R_+}\vert f''(x)\vert
           + \big(\log E(\ee^{\xi_1};\zeta>1)\big)\sup_{x\in\R_+}\vert f'(x)\vert\\
      &\quad +qm \sup_{x\in\R_+}\vert f''(x)\vert
              + m \sup_{x\in\R_+}\vert f''(x)\vert \int_\R(\ee^u-1)^2\,\Pi(\dd u)\Bigg)
       <\infty.
 \end{align*}

To finish the proof, we can resort to Proposition 4.2 of Fu and Li \cite{FL} (see also Theorem II.2.42 of Jacod and Shiryaev \cite{JS}). Since $ Z^m$ is a solution
 to the martingale problem (\ref{mm}), it is then a weak solution
 of the SDE \eqref{eqn:o3}.
\end{proof}

The last step for the existence proof is to get rid of the additional truncation by $m$ which
 was only included to validate Aldous's criterion in Lemma \ref{lem:tightness}.
We start with a simple observation.

\begin{lemma}\label{lemma_pathwise_uniqueness}
Under the conditions of Theorem \ref{thm:1} pathwise uniqueness {for non-negative solutions} holds for the truncated
 SDE (\ref{eqn:o3}) {for $t\geq 0$.}
\end{lemma}

\begin{proof}
This follows along the lines of the proof of Proposition \ref{prop:uniqueness}
 for the modified SDE (\ref{eqn:o3}).
\end{proof}

Combining Lemma \ref{lemma_pathwise_uniqueness} with Lemma \ref{lem:weakexistence}, we can construct
 a sequence $(Z^m)_{m\in \N}$ of strong solutions to the jump type SDEs (\ref{eqn:o3})
 on the same probability space.
Indeed, under pathwise uniqueness weak solutions are automatically strong solutions, see,
 e.g. Situ \cite[page 104]{Sit}, and they can be defined on the same probability space,
 since, by the adaptedness, strong solutions are deterministic functions of the driving noises
 $B$, {$\mathcal M$} and $\cN$.
Let us define by
\begin{align*}
	\tau_m:=\inf\{ t\geq 0 : Z^m_t\geq m\}
\end{align*}
the sequence of exceeding times of level $m$.

\begin{lemma}\label{lem:tau}
	The sequence $(\tau_m)_{m\in\N}$ is increasing and tends to infinity almost surely.
\end{lemma}
\begin{proof}
{As in Lemma \ref{lll} we can assume $q=0$ as the killing integral does not influence the explosion in finite time. } As long as $Z_t^m\leq m$ taking minimum with $m$ is superfluous so that $Z^m$ satisfies
 the SDE (\ref{eqn}) up to time $\tau_m$.
Hence, by Proposition \ref{prop:uniqueness}, for all $m\leq m'$, $m,m'\in\N$,
 two solutions $Z^m$ and $Z^{m'}$ are 	almost surely identical until $\tau_m$.
This implies that $\tau_m\leq \tau_{m'}$ almost surely for all $m\leq m'$ yielding that
 the limit $\lim_{m\to\infty}\tau_m$ exists almost surely.

To show that $\tau_m$ indeed converges to infinity we use an estimate of Fu and Li \cite{FL}.
By the proof of their Proposition 2.3 we want to use the upper bound
 \begin{align}\label{help19}
	   (1+m)P(\tau_m<t) \leq (1+z)\exp\left( (\log E\big(\ee^{\xi_1}\big))t\right)
                              = (1+z)(E(\ee^{\xi_1}))^t,
	    \qquad t\geq 0.
 \end{align}
To check that we are allowed to work with \eqref{help19}, conditions (2.a) and (2.b) of
 Fu and Li \cite{FL} need to be verified.
This can be done almost identical to the proof of Lemma \ref{lll} modifying the definitions
 of the functions $\sigma$ and $g_0$ given there to $\sigma\sqrt{x\wedge m}$, $x\geq 0$, and to
 $\mathbf 1_{\{xr\leq 1\}}(x\wedge m)(\ee^u-1)$, $x\in\R$, $r\geq 0$, $u\in\R$, respectively.

Let us finally check that \eqref{help19} implies $P(\lim_{m\to\infty}\tau_m = \infty)=1$.
Suppose that $\tau:=\lim_{m\to\infty}\tau_m$ (this limit exists with probability one since the sequence
 $(\tau_m)_{m\in\N}$ is monotone increasing), then, by \eqref{help19}, for all continuity points {$t\geq 0$}
 of the distribution function of $\tau$, we get
 \[
     P(\tau< t) \leq \liminf_{m\to\infty} \frac{(1+z)(E(\ee^{\xi_1}))^t}{1+m}
                          =0.
 \]
Since the distribution function of $\tau$ has only at most countably many discontinuity points,
 this yields that $P(\tau=\infty)=1$.
 \end{proof}

We can now combine the lemmas to prove strong existence of solutions to the SDE (\ref{eqn}).

\begin{proposition}\label{P:existence}
Under the conditions of Theorem \ref{thm:1} there is a strong solution to the SDE (\ref{eqn}).
\end{proposition}

\begin{proof}
To construct a strong solution to the jump type SDE (\ref{eqn}) we use the sequence of strong solutions
 $(Z^m)_{m\in\N}$ {to \eqref{eqn:o3}} and stopping times $(\tau_m)_{m\in\N}$ constructed and defined above.
For $t\in [0,\tau_1)$, we set $Z_t:=Z^1_t$ and inductively set $Z_t:=Z^m_t$ for
 $t\in [\tau_{m-1},\tau_m)$, $m\geq 2$.
This definition is consistent, since $Z^m$ satisfies the SDE (\ref{eqn})
 up to time $\tau_m$ and, due to the pathwise uniqueness for the SDE (\ref{eqn}),
 if $m'>m$, then the solutions $Z^m$ and $Z^{m'}$ coincide up to $\tau_m$ almost surely.
The construction clearly shows that $Z$ is a strong solution to the SDE (\ref{eqn})
 for all $t\geq 0$.
\end{proof}

The pathwise uniqueness serves us mainly as a tool to prove the self-similarity of solutions.

\begin{proposition}\label{prop:self_sim_index_1}
Under the conditions of Theorem \ref{thm:1} the unique strong solutions of the SDE (\ref{eqn})
 are self-similar of index $1$.
\end{proposition}

\begin{proof}
Let $c>0$ be fixed and define $\bar Z_t:=\frac{1}{c}Z_{ct}$, $t\geq 0$,
 where $(Z_t)_{t\geq 0}$ is the unique strong solution of the SDE (\ref{eqn}) with
 initial condition $z\geq 0$.
Then we obtain
	\begin{align*}
		\bar Z_t&=\frac{z}{c}
                       + \big( \log E\big(\ee^{\xi_1};\zeta>1\big)\big) t
                       + \frac{\sigma}{c} \int_0^{ct}\sqrt{Z^{}_s} \dd B_s
                      -\frac{1}{c}\int_0^{ct}\!\!\int_0^\infty\mathbf 1_{\{r Z_{s-}\leq 1\}} Z_{s-}
                              (\mathcal M-\mathcal M')(\dd s,\dd r)\\
     &\quad    + \frac{1}{c} \int_0^{ct}\int_0^\infty\int_\R
                     \mathbf 1_{\{rZ^{}_{s-}\leq 1\}} Z^{}_{s-}[\ee^u-1]
                       (\mathcal N-\mathcal N')(\dd s,\dd r,\dd u),
      \qquad t\geq 0.
	\end{align*}
By Revuz and Yor \cite[Proposition V.1.5]{RY} we have the almost sure identity
 \[
   \frac{\sigma}{c} \int_0^{ct}\sqrt{Z_s} \dd B_s
       = \sigma\int_0^{t}\sqrt{\frac{1}{c}Z_{cs}} \dd \left(\frac{1}{\sqrt{c}}B_{cs}\right),
         \qquad t\geq 0,
 \]
 since $Z$ has c\`{a}dl\`{a}g paths so that $P\big(\int_0^t Z_s\,\dd s<\infty\big)=1$ for all $t\geq 0$.
 Next, we also use the analogue almost sure identities
 \begin{align*}
          \int_0^{ct}&\int_0^\infty
                     \mathbf 1_{\{rZ_{s-}\leq 1\}} Z^{}_{s-}
                       (\mathcal M-\mathcal M')(\dd s,\dd r)\\
            & = \int_0^t\int_0^\infty
                     \mathbf 1_{\{r\frac{1}{c}Z_{(cs)-}\leq 1\}} Z_{(cs)-}
                       (\mathcal M-\mathcal M')(c\dd s,c^{-1}\dd r),
               \qquad t\geq 0,
 \end{align*}
 and
  \begin{align*}
          \int_0^{ct}&\int_0^\infty\int_\R
                     \mathbf 1_{\{rZ_{s-}\leq 1\}} Z^{}_{s-}[\ee^u-1]
                       (\mathcal N-\mathcal N')(\dd s,\dd r,\dd u)\\
            & = \int_0^t\int_0^\infty\int_\R
                     \mathbf 1_{\{r\frac{1}{c}Z_{(cs)-}\leq 1\}} Z_{(cs)-}[\ee^u-1]
                       (\mathcal N-\mathcal N')(c\dd s,c^{-1}\dd r,\dd u),
               \qquad t\geq 0.
 \end{align*}
We do not give the justifications for these simple facts and refer to the arguments of the proof
 of Proposition \ref{prop:reversed}, Steps 4c), 4d) in a more delicate context.
Motivated by the above three identities, we define the Wiener process
 $\bar B_t:=\frac{1}{\sqrt{c}}B_{ct}$, $t\geq 0$, and {the independent Poisson random measures  $\bar\cN$ on $(0,\infty)\times (0,\infty)\times \R$
 by $\bar\cN(\dd s,\dd r,\dd u):=\cN(c\dd s,c^{-1}\dd r,\dd u)$ and  $\bar {\mathcal M}$ on $(0,\infty)\times (0,\infty)$
 by $\bar{\mathcal M}(\dd s,\dd r):={\mathcal M}(c\dd s,c^{-1}\dd r)$.

 It follows directly from the definition of a Poisson random measure that
 $\bar{\mathcal N}$ is a Poisson random measure with the
 same intensity measure as $\mathcal N$ and $\bar{\mathcal M}$ is a Poisson random measure with the
 same intensity measure as $\mathcal M$.
With these definitions, the above calculation leads to
	\begin{align*}
     \bar Z_t   &= \frac{z}{c} + \left( \log E(\ee^{\xi_1};\zeta>1) \right) t
          +\sigma \int_0^{t} \sqrt{\bar Z_{s}} \dd \bar B_{s}
                   - \int_0^{t}\int_0^\infty
                  \mathbf 1_{\{r\bar Z_{s-}\leq 1\}} \bar Z_{s-}
                  (\bar{\mathcal M} - \bar{\mathcal M}')(\dd s,\dd r)\\
          &\quad+ \int_0^{t}\int_0^\infty\int_\R
                  \mathbf 1_{\{r\bar Z_{s-}\leq 1\}} \bar Z_{s-}[\ee^u-1]
                  (\bar{\mathcal N} - \bar{\mathcal N}')(\dd s,\dd r,\dd u),
               \quad t\geq 0.
	\end{align*}}
Hence, $\bar Z$ is a strong solution to (\ref{eqn}) with initial condition $\frac{z}{c}$ so that pathwise uniqueness (which implies weak uniqueness) for (\ref{eqn}) proved in Proposition \ref{prop:uniqueness} implies the claim.
\end{proof}

Finally, we prove the strong Markov property of solutions to (\ref{eqn}).

\begin{proposition}\label{SM}
Under the conditions of Theorem \ref{thm:1} the unique strong solution of the jump type SDE (\ref{eqn})
 is {a strong Markov process.}
\end{proposition}

\begin{proof}
In the following for all $0<\eps<m$ and $z\geq 0$, let us denote by $\P^{\eps,m}_z$, $\P^m_z$, $\P_z$ the laws of
 the unique solutions to (\ref{eqn:o2}), (\ref{eqn:o3}), (\ref{eqn}) started in $z\in\R_+$ on $(\mathbb D,\mathcal D)$.

\textbf{Step 1 (measurability of $z\mapsto\P_z$):}
We first prove that the martingale problem corresponding to {the infinitesimal generator of}
 the {strong} solutions $Z^{\eps,m}$ {to (\ref{eqn:o2})} is well-posed.
 {The well-posedness} for deterministic initial conditions $z\geq 0$
 {follows from part (a) of Theorem \ref{thm:1},}
 and can be extended to {general initial distributions (i.e. probability measures on $\R_+$)}:
 by Problem 49 on page 273 of Ethier and Kurtz \cite{EK} it suffices to verify the conditions of their Lemma 4.5.13.
Namely, we have to check that the domain of the infinitesimal generator of $\cA^{\eps,m}$ contains an algebra that
  separates points and vanishes nowhere (i.e. not all of its elements simultaneously vanish at a point) and that
  for each compact set $K\subseteq \R_+$ and $\eta>0$ there exist a sequence of compact subsets $K_n\subseteq \R_+$,
  $K\subseteq K_n$, $n\in\N$, and smooth functions $f_n:\R_+\to\R$, $n\in\N$, with compact support such that for
  $F_n:=\{z\in\R_+ : \inf_{x\in K_n}\vert x-z\vert\leq \eta\}$ we have
  \begin{align}\label{EK1}
    &\beta_{n,\eta}:=\inf_{y\in K} f_n(y) - \sup_{y\in\R_+\setminus F_n} f_n(y) >0,\qquad n\in\N,\\ \label{EK2}
    &\lim_{n\to\infty}\beta_{n,\eta}^{-1}\sup_{y\in F_n} (\cA^{\eps,m} f_n)^{-}(y)=0,\\ \label{EK3}
    &\lim_{n\to\infty}\beta_{n,\eta}^{-1} \Big( \sup_{x\in\R_+}\vert f_n(x)\vert - \inf_{y\in K}f_n(y)\Big)=0,
  \end{align}
  where $(\cA^{\eps,m} f)^{-}$ denotes the negative part of $\cA^{\eps,m} f$.
First note that the set of bounded twice continuously differentiable functions on $\R_+$ with compact support
 and bounded first and second derivatives serves as an appropriate choice for such an algebra.
Next let $k\in\N$ be such that $K\subseteq [0,k)$, $K_n:=[0,k+n]$, $n\in\N$, and let us choose the functions
 $f_n$, $n\in\N,$ as in Example 4.5.14 in Ethier and Kurtz \cite{EK}, namely $f_n:\R_+\to\R$ is a smooth function
 with compact support such that
 \begin{itemize}
   \item $f_n(x) = 1 + \log(1+ (k+n+\eta)^2) - \log(1+ x^2)$ if $0\leq x\leq k+n+\eta$,
   \item $f_n(x)\in[0,1]$ if $x>k+n+\eta$,
   \item $\sup_{x\in\R_+\setminus F_n}\vert f_n'(x)\vert <\infty$ and
         $\sup_{x\in\R_+\setminus F_n}\vert f_n''(x)\vert <\infty$. \end{itemize}
By Example 4.5.14 in Ethier and Kurtz \cite{EK}, \eqref{EK1} and \eqref{EK3} hold.
By (\ref{ka}) we have
\begin{align*}
  \sup_{y\in F_n} (\cA^{\eps,m} f_n)^{-}(y) \leq \sup_{y\in\R_+} (\cA^{\eps,m} f_n)(y)<\infty,
   \qquad n\in\N,
\end{align*}
 further, by the definition of $f_n$,
 \[
   \beta_{n,\eta}\geq \log(1+ (k+n+\eta)^2) \to \infty \quad \text{as $n\to\infty$,}
 \]
 which yields \eqref{EK3}. Then Theorem 4.4.6 of Ethier and Kurtz \cite{EK} implies measurability of $z\mapsto \P^{\eps,m}_z$ mapping $(\R_+,\mathcal B(\R_+))$ to $(\mathcal P(\mathbb D),\rho)$,
 where $\mathcal P(\mathbb D)$ is the set of all probability measures on $(\D,\cD)$ equipped
 with the Prokhorov metric $\rho$.\\
Since {for all $z\geq 0$, $\P^{\eps,m}_z$ converges to $\P^m_z$} with respect to the Prokhorov metric
 {$\rho$ as $\eps\downarrow 0$},
 the map $z\mapsto \P_z^m$ is measurable as pointwise limit of a sequence of measurable maps.
Finally, notice that our pathwise construction of $Z$ via $Z^m$ {given in the proof of Proposition \ref{P:existence}}
 implies in particular that $Z^m$ converges weakly to $Z$ {as $m\to\infty$ in the Skorokhod's topology of $\D$.}
{Indeed, by the construction, if $m'\geq m$, then $Z^{m'}_t=Z^m_t$ for $t\in[0,\tau_m)$,
 and, by Lemma \ref{lem:tau}, $(\tau_m)_{m\in\N}$ is increasing and tends to infinity almost surely.
Hence, $Z^m$ converges to $Z$ as $m\to\infty$ in the local uniform topology on $\D$.
Since this topology is weaker than the Skorokhod's topology (see, e.g. Jacod and Shiryaev \cite[Proposition VI.1.17]{JS}),
 we find that $Z^m$ converges weakly to $Z$ as $m\to\infty$ in the Skorokhod's topology of $\D$.
Using again that the pointwise limit of a sequence of measurable mappings is measurable
 we have the measurability of the map $z\mapsto \P_z$.}

\textbf{Step 2 (uniqueness for general initial {distributions}):} Well-posedness for
 {the martingale problem corresponding to the infinitesimal generator of the strong solutions $Z$ to (\ref{eqn})
  with general non-negative initial distributions}
  follows then from well-posedness for deterministic initial conditions
 $z\in \R_+$ combined with the measurability of the first step (see for instance Lemma IX.4.4 of Jacod and Shiryaev \cite{JS}).
	
\textbf{Step 3 (strong Markov property):} Finally, the strong Markov property {for the unique
 strong solutions to the SDE (\ref{eqn:o2})} follows from the first two
steps for instance by Theorem 4.4.2 of Ethier and Kurtz \cite{EK}.
\end{proof}

We now come to the arguments needed to turn the heuristic derivation of the jump type SDE (\ref{eqn}) given
 in the introduction into mathematics.
In order to determine the Lamperti transformed L\'evy process corresponding to the
 unique strong solution $Z$ absorbed at zero, we show that applying Lamperti's transformation to the L\'evy process $\xi$
 associated to the SDE (\ref{eqn}) leads to a weak solution of (\ref{eqn}) up to the first hitting time of zero.

\begin{proposition}\label{prop:reversed}
{Under the conditions of Theorem \ref{thm:1} let}
 $Z$ be the pathwise unique strong solution to (\ref{eqn}).
Then the Lamperti transformed L\'evy process $\xi$ {of the pssMp $Z^\dag$ of index $1$}
 has L\'evy triplet $(\gamma,\sigma^2,\Pi)$ {and is} killed at rate $q\geq 0$.
\end{proposition}

\begin{proof}

By Proposition \ref{prop:self_sim_index_1}, the unique strong solution to (\ref{eqn}) is self-similar of self-similarity index $1$ so that the absorbed process $Z^\dag$ is a pssMp of self-similarity index $1$.
 Now let $\xi$ be a L\'evy process with
 L\'evy triplet $(\gamma,\sigma^2,\Pi)$ and killing rate $q$ and, for some $z>0$, define
 \begin{align}\label{Y}
	    Y_t := z \exp\left(\xi_{\tau(tz^{-1})}\right),\quad t\geq 0,
 \end{align}
 where $(\tau(t))_{t\geq 0}$ is the (extended real-valued) continuous, increasing inverse
 of $I_t=\int_0^t\ee^{\xi_s}\,\dd s$, $t\geq 0$, with the convention $Y_t:=0$ for $t\geq S_0$,
 where $S_0:=\inf\{t\geq 0 : Y_t=0\}$.
To prove the statement of the proposition, by pathwise uniqueness (which implies well-posedness
 of the corresponding martingale problem), it is then enough to check that $(Y_t)_{{t\in[0,S_0)}}$
 is a weak solution of the SDE \eqref{eqn} with initial condition $z$ up to the first hitting time of $0$.

Since the killing only occurs for the first
 jump of $Y$ down to zero at $S_0$,
 which - as we explained in the derivation of (\ref{eqn}) - resulted in the killing integral of (\ref{eqn})
 we assume from now on  $q=0$.

 We will use throughout the proof that $S_0<\infty$ almost surely if and only if $\xi$ drifts to $-\infty$ so that in this case the infinite interval $[0,\infty)$ is compressed via $\tau$ to the finite interval $[0,S_0)$. More formally, this says that almost surely $\tau(tz^{-1})<\infty$ for $t<S_0$ with  $\lim_{t\uparrow S_0}\tau(tz^{-1})=\infty$.
\smallskip

\textbf{Step 1 (integrals are well-defined):}
First we check that with the definition (\ref{Y}), $Y$ can be a solution to (\ref{eqn}) in the sense that
 the integrals in the SDE \eqref{eqn} can be defined as local martingales.

\smallskip

\textbf{Step 1a (localization sequence):}
We show that there exists a sequence of stopping times $(\delta_m)_{m\in\N}$ {defined via $Y$}
 such that $P(\lim_{m\to\infty}\delta_m=\infty)=1$.
Let us define the increasing sequence of stopping times $\delta_m:=\inf\{t\geq 0:Y_t\geq m\}$, $m\in\N$.
Since $\delta_m$ increases there exists a random variable $C$ such that
 $\lim_{m\to\infty}\delta_m = C$ almost surely.
We aim at proving that $P(C=\infty)=1$.
First note that due to the aforementioned properties of $\tau$, $(Y_t)_{t\geq 0}$ has almost surely
 c\`{a}dl\`{a}g paths.
In what follows we distinguish the cases according to $\xi$ drifting to $-\infty$, $+\infty$ or oscillating.

If $\xi$ drifts to $+\infty$ or oscillates, then $S_0=\infty$ almost surely
 and hence $\lim_{t\to\infty} \tau(tz^{-1})=\infty$ so that $(\xi_{\tau(tz^{-1})})_{t\geq 0}$ drifts to $+\infty$ or
 oscillates.
Let us suppose that $P(C<\infty)>0$ and let $\omega\in\Omega$ be such that $C(\omega)<\infty$.
Using that $\delta_m(\omega)\leq C(\omega)$, $m\in\N$, by the definition of an infimum, there exists a sequence
 $(t_m)_{m\in\N}$ (depending on $\omega$) such that $Y_{t_m}(\omega)\geq m$, $m\in\N$, and
 $t_m<C(\omega)$, $m\in\N$ (in the case of $\delta_m(\omega)<C(\omega)$, $m\in\N$), or
 $\lim_{m\to\infty} t_m(\omega) = C(\omega)$ (in the case of $\delta_{m_0}(\omega)=C(\omega)$ for some $m_0\in\N$).
Hence, since $\tau$ is increasing and continuous, $\tau(t_mz^{-1})(\omega)\leq \tau(Cz^{-1})(\omega)$, $m\in\N$,
 or $\lim_{m\to\infty}\tau(t_mz^{-1})(\omega)=\tau(Cz^{-1})(\omega)$.
Further, since $Y_{t_m}(\omega)\geq m$, $m\in\N$, we have $\lim_{m\to\infty}Y_{t_m}(\omega)=\infty$
 and then $\lim_{m\to\infty}\xi_{\tau(t_m z^{-1})}(\omega)=\infty$.
Then it would imply that $\xi$ would not be bounded on the interval $[0,\tau(Cz^{-1})(\omega)+1]$.
This is a contradiction, since, by $C(\omega)<S_0(\omega)=\infty$, we have $\tau(Cz^{-1})(\omega)<\infty$
 implying that the interval $[0,\tau(Cz^{-1})(\omega)+1]$ is finite.
Hence $P(C=\infty)=1$ if $\xi$ drifts to $+\infty$ or it oscillates.

If $\xi$ drifts to $-\infty$, then $S_0$ is almost surely finite and
 $\lim_{t\uparrow S_0} Y_t = \lim_{t\uparrow S_0} z\ee^{\xi_{\tau(tz^{-1})}}=0$
 since $\lim_{t\uparrow S_0}\tau(tz^{-1})=\infty$.
Recalling also that, by convention, $Y_t=0$ for $t\geq S_0$ and $Y$ has c\`{a}dl\`{a}g paths,
 we find that the trajectories of $(Y_t)_{t\geq 0}$ are bounded from above implying
 that $\delta_m$ tends to infinity as $m\to\infty$.
Hence $P(C=\infty)=1$ also in the case of $\xi$ drifting to $-\infty$.
\smallskip

\textbf{Step 1b (the integral with respect to $B$):}
To define the integral $\int_0^t\sqrt{Y_s}\,\dd B_s$ we use the localized estimate
 \begin{align}\label{h}
   E\left(\int_0^{t\wedge \delta_m} Y_s\,\dd s\right)
      \leq m E(t\wedge \delta_m) \leq tm <\infty,
  \qquad m\in\N, \; t\geq 0,
 \end{align}
 and Definition 2.4 in Ikeda and Watanabe \cite[page 57]{IW}.\\

\textbf{Step 1c (the integral with respect to $\cN-\cN'$):}
To define the Poissonian integral
 \[
   \int_0^{t }\int_0^\infty\int_\R
        \mathbf 1_{\{rY_{s-}\leq 1\}} Y_{s-}(\ee^u-1)
                             (\mathcal N-\mathcal N')(\dd s,\dd r,\dd u),
 \]
 we use
 \begin{align*}
 	&\E\left(   \int_0^{t\wedge\delta_m}\int_0^\infty\int_\R
       \mathbf 1_{\{rY_{s-}\leq 1\}}(Y_{s})^2(\ee^u-1)^2\,\dd s\,\dd r\,\Pi(\dd u)\right)\\
    &\quad = \E\left( \int_0^{t\wedge\delta_m} Y_s\,\dd s\right) \int_\R (\ee^u-1)^2\,\Pi(\dd u)<\infty,
 \end{align*}
 which follows by (\ref{h}) and that $E(\ee^{2\xi_1})<\infty$ {(which holds since $\xi$ is spectrally negative).}
Hence, the localized definition of Ikeda and Watanabe \cite[page 62]{IW} can be applied.

\smallskip

\textbf{Step 2 (SDE for \textbf{\mvb $z\exp(\xi)$}):}
First let us define $\eta_t:=z\exp(\xi_t)$, $t\geq 0$. Applying
 It\=o's formula (see, e.g. Ikeda and Watanabe \cite[Chapter II, Theorem 5.1]{IW})
 to the L\'evy-It\=o representation \eqref{li} of $\xi$, we have
 \begin{align}\label{Ito_exp}
  \begin{split}
   \eta_t
     &= z + \gamma \int_0^t \eta_s\dd s
          + \sigma \int_0^t \eta_s\dd B_s
          + \frac{\sigma^2}{2} \int_0^t \eta_s\dd s\\
     &\quad + \int_0^t\int_{|u|\leq 1} ( z\ee^{\xi_{s-} + u}  - z\ee^{\xi_{s-}} )
               (\mathcal{N}_0 - \mathcal{N}_0')(\dd s,\dd u)  \\
     &\quad  + \int_0^t\int_{|u|>1} ( z\ee^{\xi_{s-} + u}  - z\ee^{\xi_{s-}} )
                      \mathcal{N}_0 (\dd s,\dd u)\\
     &\quad  + \int_0^t\int_{|u|\leq 1}
              ( z\ee^{\xi_s + u}  - z\ee^{\xi_s} -uz\ee^{\xi_s})
             \,\dd s\,\Pi(\dd u) \\
     &= z + \left( \gamma + \frac{\sigma^2}{2} + \int_{|u|\leq 1}(\ee^u - 1 - u)\,\Pi(\dd u)\right)
             \int_0^t \eta_s\dd s
             + \sigma\int_0^t \eta_s \dd B_s\\
    & \phantom{=\;\;}
             + \int_0^t\int_{|u|\leq 1} \eta_{s-}(\ee^u-1)
                    (\mathcal{N}_0 - \mathcal{N}_0')(\dd s,\dd u)\\
     &\phantom{=\;\;}  + \int_0^t\int_{|u|>1} \eta_{s-}(\ee^u-1)
                      \mathcal{N}_0 (\dd s,\dd u) , \quad t\geq 0.
  \end{split}
 \end{align}
\smallskip

\textbf{Step 3 (some preparations):}
In order to deduce from (\ref{Ito_exp}) that the time-changed process
 $Y_t=\eta_{\tau(tz^{-1})}$, $t\in[0,S_0)$, is a solution of (\ref{eqn}) {with $q=0$} up to first hitting zero
 we need some preparations.
\smallskip

\textbf{Step 3a (the local martingale $W$ and the Poisson random measure $\mathcal P$):}
We denote by $(t_n,\Delta_n)_{n\in\N}$ an arbitrary labeling of the pairs
 associated to jump times and jump sizes of the time-changed L\'evy process
 $\left(\xi_{\tau(tz^{-1})} \right)_{t\in[0,S_0)}$.
By an enlarging procedure, we can assume that we are given a Wiener process $(\bar B_t)_{t\geq 0}$
 and an independent Poisson random measure $\cQ$ on $(0,\infty)\times (0,\infty)\times\R$ with intensity
 measure $\dd s\otimes \dd r\otimes\Pi(\dd u)$ both defined independently of $(B_{\tau(tz^{-1})})_{t\in[0,S_0)}$
 and $(t_n,\Delta_n)_{n\in\N}$ on a common filtered probability space
 $(\Omega',\mathcal H,(\mathcal H_t)_{t\geq 0},P')$.
Note that $\cG_{\tau(tz^{-1})}\subseteq \cH_t$, $t\leq 0<S_0$.
To keep the notation simple, we still use $P$ and $E$ for the (extended) probability measure and the expectation
 with respect to it, {respectively.}
Additionally, we choose an independent sequence of random variables $(R_n)_{n\in\N}$ uniformly distributed
 on $(0,1)$ such that $R_n$ is $\cH_{t_n}$-measurable and independent of $\cH_{t_n-}$
 (the existence of such a sequence can be assumed by the above mentioned enlargement of filtration).
Since $(B_{\tau(tz^{-1})})_{t\in[0,S_0)}$ is a continuous local martingale (not necessarily a Brownian motion)
 with respect to the filtration $(\cH_t)_{t\in[0,S_0)}$, we can define
 \[
   W_t := \int_0^t \sqrt{Y_s} \mathbf 1_{\{Y_s\ne 0\}}
                      \,\dd B_{\tau(sz^{-1})}
                     + \int_0^t \mathbf 1_{\{Y_s= 0\}}\,\dd \bar B_s,
                     \qquad t\geq 0.
 \]
Note that if $\xi$ drifts to $-\infty$, then $S_0<\infty$ almost surely and $\tau(sz^{-1})=\infty$
 for $s\geq S_0$.
Hence $B_{\tau(sz^{-1})}$ is not defined in this case, however, by convention, $Y_s=0$ for $s\geq S_0$
 and then in this case the first integral in the definition of $W$ is understood to be $0$.
We also define a new point measure by
 \begin{align*}
   \cP(A_1\times A_2\times A_3)
      & :=\sum_{n=1}^\infty{\mathbf 1_{A_1\times A_2\times A_3}((t_n,R_n/Y_{t_n-},\Delta_n))}  \\
         &\quad+\int_{A_1}\int_{A_2}\int_{A_3}(\mathbf 1_{\{rY_{s-}>1 \}} + \mathbf 1_{\{Y_{s-}=0\}} )
            \,\cQ(\dd s,\dd r,\dd u)
 \end{align*}
 for all $A_1,A_2\in\cB((0,\infty))$ and $A_3\in\cB(\R)$.
Here we note that we never divide by zero ($Y_{t_n-}\ne 0$, $n\in\N$), since $t_n<S_0$, $n\in\N$,
 and hence $\cP$ is well-defined.

\smallskip

\textbf{Step 3b ($W$ is a Brownian motion):}
First observe that $(W_t)_{t\geq 0}$ is a continuous local martingale
 with respect to the filtration $(\cH_t)_{t\geq 0}$, since by localizing
 with the sequence $(\delta_m)_{m\in\N}$ defined in Step 1a
 \begin{align*}
&\quad   E\left(\int_0^{t\wedge \delta_m} Y_s \mathbf 1_{\{Y_s\ne 0\}}
                      \,\dd \tau(sz^{-1})
                     + \int_0^{t\wedge \delta_m} \mathbf 1_{\{Y_s= 0\}}\,\dd s\right)\\
  & = E\left(\int_0^{t\wedge \delta_m} \mathbf 1_{\{Y_s\ne 0\}} \,\dd s
                     + \int_0^{t\wedge \delta_m} \mathbf 1_{\{Y_s= 0\}}\,\dd s\right) \leq t,
   \quad t\geq 0,\;\; m\in\N,
 \end{align*}
 so that Definition 2.4 in Ikeda and Watanabe \cite[page 57]{IW} can be used.
The equality above is justified by
 \begin{align}\label{help23}
   \tau'(t) = \frac{z}{Y_{tz}},\qquad t\in\big[0,S_0/z\big),
 \end{align}
 which easily follows, since, by the definition of the time-change $\tau$,
 \[
   \int_0^{\tau(t)} \exp(\xi_s)\,\dd s = t, \quad t\in\big[0,S_0/z),
 \]
 and differentiating both sides with respect to $t$ gives
 \[
   \exp(\xi_{\tau(t)}) \tau'(t) = 1, \quad t\in\big[0,S_0/z\big).
 \]
Furthermore, using the bilinearity of quadratic variation and the independence of the processes
 $(\bar B_t)_{t\geq 0}$ and $(B_{\tau(tz^{-1})})_{t\in[0,S_0)}$,
 the quadratic variation process of the $(\mathcal H_t)$-local martingale $W$ is
 \begin{align*}
  \langle W\rangle_t
    = \int_0^t Y_s \mathbf 1_{\{Y_s\ne 0\}} \,\dd \tau(sz^{-1})
       + \int_0^t \mathbf 1_{\{ Y_s=0\}}\,\dd s
    = \int_0^t \mathbf 1_{\{Y_s\ne 0\}}\,\dd s
      + \int_0^t \mathbf 1_{\{Y_s=0\}}\,\dd s
    = t
 \end{align*}
 for $t\geq 0$.
Here we call the attention to the fact that $\int_0^t \sqrt{Y_s} \mathbf 1_{\{Y_s\ne 0\}} \,\dd B_{\tau(sz^{-1})}$,
 $t\in[0,S_0)$, is not a square-integrable martingale in general, only a continuous local martingale,
 however, its quadratic variation process equals $\int_0^t Y_s \mathbf 1_{\{Y_s\ne 0\}} \,\dd \tau(sz^{-1})$,
 $t\in[0,S_0)$, see, e.g., Karatzas and Shreve \cite[page 147]{KS}.
Finally, by L\'evy's theorem, this shows that $(W_t)_{t\geq 0}$ is a $(\cH_t)$-Wiener process.

\smallskip

\textbf{Step 3c ($\mathcal P$ is a Poisson random measure):}
We show that $\mathcal P$ is a Poisson random measure on $(0,\infty)\times (0,\infty)\times \R$ with
 intensity measure $\dd s\otimes \dd r\otimes\Pi(\dd u)$.
Since $\mathcal P$ was defined to be an optional random measure, by Jacod and Shiryaev \cite[Theorems II.1.8 and II.4.8]{JS},  it is enough to show that
 \begin{align}\label{help1}
  \begin{split}
    &E\left(\int_0^\infty \int_0^\infty \int_\R H(s,r,u)\,\cP(\dd s,\dd r,\dd u)\right)
= E\left(\int_0^\infty \int_0^\infty \int_\R H(s,r,u)\,\dd s\,\dd r\,\Pi(\dd u)\right)
  \end{split}
 \end{align}
 for every non-negative predictable function $H$ on $\Omega\times (0,\infty)\times(0,\infty)\times\R$.
Here predictability means measurability with respect to the $\sigma$-algebra $\cA\otimes \cB((0,\infty)\times\R)$,
 where $\cA$ is the $\sigma$-algebra on $\Omega\times (0,\infty)$ generated by all left-continuous adapted
 processes (considered as mappings on $\Omega\times (0,\infty)$).
Note also that \eqref{help1} includes the case when $H$ is not integrable with respect to $\cP$
 meaning that both sides of \eqref{help1} are infinite in this case.
By the definition of $\mathcal P$ we can write
 \begin{align}\label{help21}
  \begin{split}
   E&\left(\int_0^\infty \int_0^\infty \int_\R H(s,r,u)\,\cP(\dd s,\dd r,\dd u)\right)\\
    &= E\left(\sum_{n=1}^\infty
         H(t_n,R_n/Y_{t_n-},\Delta_n)\right)\\
    &\quad  +  E\left(\int_0^\infty \int_0^\infty \int_\R H(s,r,u)
                   (\mathbf 1_{\{rY_{s-}>1\}} + \mathbf 1_{\{Y_{s-}=0\}} )
                     \,\cQ(\dd s,\dd r,\dd u)\right).
  \end{split}
 \end{align}
To express the first summand we apply Theorem II.1.8 of Jacod and Shiryaev \cite{JS} to the
 non-negative predictable function $$\widetilde H(s,r,u):=H(s,r/Y_{s-},u)
 ,\quad s>0,r>0,u\in\R,$$
 and the Poisson random measure on $(0,\infty)\times(0,\infty)\times\R$ defined by
 \[
   \widetilde\cP(A_1\times A_2\times A_3)
      := \sum_{n=1}^\infty {\mathbf 1_{A_1\times A_2\times A_3}((t_n,R_n,\Delta_n))},
      \qquad A_1,A_2\in\cB((0,\infty)),\;\; A_3\in\cB(\R),
 \]
 to obtain
 \begin{align*}
	 	 E\left(\sum_{n=1}^\infty H(t_n,R_n/Y_{t_{n-}},\Delta_n) \right)
		 &=\E\left(\int_0^\infty \int_0^\infty \int_\R \widetilde H(s,r,u)\,
                   \widetilde\cP(\dd s,\dd r,\dd u)\right)\\
	     	 &= E\left(\int_0^{S_0} \int_0^1 \int_\R H(s,r/Y_{s-},u)
                       \,\dd \tau({sz^{-1})} \,\dd r\,\Pi(\dd u) \right),
 \end{align*}
 where the second equality holds since, by construction, the compensator measure of $\widetilde\cP$
 is  $$\mathbf 1_{(0,S_0)}(s)\mathbf 1_{(0,1)}(r)\dd \tau(sz^{-1})\,\dd r\,\Pi(\dd u).$$
Utilizing \eqref{help23} and a change of variable in the second coordinate of $H$,
 we can further simplify the right-hand side to
 \begin{align*}%\label{help_P_1}
% \begin{split}
&E\left(\int_0^{S_0} \int_0^1 \int_\R \frac{1}{Y_{s-}} H(s,r/ Y_{s-},u) \, \dd s\,\dd r\,\Pi(\dd u) \right)\\
&\quad  = E\left(\int_0^{S_0} \int_0^{1/Y_{s-}}\int_\R  H(s,r,u) \,\dd s\,\dd r\,\Pi(\dd u) \right).
%\end{split}
 \end{align*}
Similarly, applying Theorem II.1.8 of Jacod and Shiryaev \cite{JS} to $\cQ$, the second summand of
 {the right hand side of} \eqref{help21} equals
  \begin{align*}%\label{help_P_2}
 % \begin{split}
   	E&\left(\int_0^\infty \int_0^\infty \int_\R
        H(s,r,u)(\mathbf 1_{\{rY_{s-}>1\}} + \mathbf 1_{\{Y_{s-}=0\}} )\,\cQ(\dd s,\dd r,\dd u)\right) \\
     &= E\left(\int_0^{S_0} \int_{1/Y_{s-}}^\infty\int_\R  H(s,r,u) \, \dd s \,\dd r\,\Pi(\dd u) \right)
         + E\left(\int_{S_0}^\infty \int_0^\infty\int_\R  H(s,r,u) \, \dd s \,\dd r\,\Pi(\dd u) \right).
  %  \end{split}
   \end{align*}
Adding {the right hand sides of the above two equalities, by \eqref{help21}, we have} \eqref{help1}.

\smallskip

\textbf{Step 3d ($W$ and $\mathcal P$ are independent):}
This follows from Ikeda and Watanabe \cite[Chapter II, Theorem 6.3]{IW}, since
 the compensator measure of $\cP$, by Step 3c, is deterministic.

\smallskip

 \textbf{Step 4 (deriving the SDE (\ref{eqn}) - plugging-in):}
We first observe that, by \eqref{Ito_exp}, the time-changed process
 $Y_t= \eta_{\tau(tz^{-1})}$, $t\in[0,S_0)$, satisfies the integral equation
 \begin{align}\label{eqn:22}\begin{split}
  Y_t&= z +  \left(\gamma+\frac{\sigma^2}{2}+\int_{|u|\leq 1}(\ee^u-1-u)\Pi(\dd u)\right)
              \int_0^{\tau({tz^{-1}})} \eta_s\dd s + \sigma\int_0^{\tau(tz^{-1})} \eta_s \dd B_s\\
	  &\quad + \int_0^{\tau(tz^{-1})}\int_{|u|\leq 1}\eta_{s-}(\ee^u-1)(\mathcal{N}_0 - \mathcal{N}_0')(\dd s,\dd u)\\
      &\quad + \int_0^{\tau(tz^{-1})}\int_{|u|> 1}\eta_{s-}(\ee^u-1)\mathcal{N}_0 (\dd s,\dd u)\end{split}
 \end{align}
 for $t<S_0$.
In the following we replace the {last} four summands on an event having probability one
 in such a way that (\ref{eqn:22}) becomes {(\ref{ab})} for $t<S_0$ driven by $W$ and $\mathcal P$
 (i.e., replacing $B$ and $\cN_0$ by $W$ and $\cP$, respectively).

\smallskip

\textbf{Step 4a (transforming the drift):}
Using \eqref{help23}, the first summand simplifies to the desired form, since by a change of variable, we have
	\begin{align}\label{help2}
   		\int_0^{\tau({tz^{-1})}}  \eta_s\dd s
     		= \int_0^{ \tau({tz^{-1})}} z\exp(\xi_s)\dd s
     		=t, \quad t<S_0.
	\end{align}

\smallskip

\textbf{Step 4b (transforming the Brownian integral):}
First we check that almost surely
 \begin{align}\label{help3}
   \int_0^{ \tau({tz^{-1})}} \eta_s \dd B_s
  	 = \int_0^t\eta_{ \tau({uz^{-1})} } \dd B_{\tau({uz^{-1})}}
     = \int_0^t  Y_u \,\dd B_{\tau({uz^{-1})}},
    \qquad t\in[0,S_0).
 \end{align}
Note that, for all $t\geq 0$,
 \begin{align}\label{help_YOR}
  \int_0^t \eta_s^2\dd\langle B\rangle_s
    = \int_0^t {z^2}\ee^{2\xi_s}\dd s < +\infty
    \qquad \text{a.s.,}
 \end{align}
 since the L\'evy process $\xi$ does not explode in finite time.
We distinguish the cases according to $\xi$ drifting to $+\infty$, $-\infty$ or oscillating.
If $\xi$ drifts to $+\infty$ or oscillates, then $S_0=\infty$ almost surely and hence $\tau(tz^{-1})$ is finite
 almost surely for all $t\geq 0$.
By \eqref{help_YOR}, we can use Revuz and Yor \cite[Proposition V.1.5]{RY}
 with the following choices
 \[
  C_t:=\tau(tz^{-1}),  \qquad H_t:=\eta_t,  \qquad X_t:= B_t,
 \]
 to infer \eqref{help3}.
If $\xi$ drifts to $-\infty$, then $S_0<\infty$ almost surely and $C_t=\infty$ almost surely
 for $t\geq S_0$, which yield that one can not use directly Revuz and Yor \cite[Proposition V.1.5]{RY}.
However, we can proceed as follows:
For all $\varepsilon>0$, let us define $S_\varepsilon:=\inf\{t\geq 0 : Y_t<\varepsilon\}$.
Since we assumed that $\xi$ drifts to $-\infty$ we may use that
 $P(Y_{S_0-}=0)=1$ to deduce that
 $P(\lim_{\varepsilon\downarrow 0} S_\varepsilon =S_0)=1$.
Next, we define the modified time-change $\tau_\varepsilon(t):=\tau((t\wedge S_\varepsilon)z^{-1})$
 for any $t\geq 0$.
Then $\tau_\varepsilon$ is a stopping time in the sense of Revuz and Yor \cite[Definition V.1.2]{RY} and
 $\tau_\varepsilon(t)<\infty$ for all $t\geq 0$, since it remains constant for $t\geq S_\eps$.
By \eqref{help_YOR}, we can use Revuz and Yor \cite[Proposition V.1.5]{RY} to infer that for all $\varepsilon>0$
 \[
   \int_0^{ \tau_\varepsilon(t)} \eta_s \dd B_s
  	 = \int_0^t\eta_{ \tau_\varepsilon(s) } \dd B_{ \tau_\varepsilon(s) },
    \qquad t\geq 0.
 \]
Since $\tau_\varepsilon(t)=\tau(tz^{-1})$, $0\leq t\leq S_\varepsilon$, we have
 \[
   \int_0^{ \tau(tz^{-1})} \eta_s \dd B_s
      = \int_0^t \eta_{\tau(sz^{-1})} \dd B_{\tau(sz^{-1})},
      \qquad 0\leq t\leq S_\varepsilon.
 \]
Using that $P(\lim_{\varepsilon\downarrow 0}S_\varepsilon = S_0)=1$, we have \eqref{help3} also
 in the case of $\xi$ drifting to $-\infty$.
Then, by the definition of $W$ and using that $Y_{s-}\ne 0$, $0<s<S_0$, we have the almost sure identity
 \begin{align*}
   \int_0^{\tau(tz^{-1})}\eta_s \dd B_s
       = \int_0^t  Y_s  \,\dd B_{\tau({sz^{-1})}} =  \int_0^t \sqrt{ Y_s}\dd W_s ,  \qquad t\in[0,S_0).
 \end{align*}
This completes the arguments for the Brownian integral.
\smallskip

\textbf{Step 4c (transforming the integral for large jumps):} Next, we check that almost surely
  \begin{align}\label{help_large_jumps}
   \begin{split}
    &\int_0^{ \tau({tz^{-1})}}
     			\int_{|u|>1} \eta_{s-}(\ee^u-1)\mathcal{N}_0 (\dd s,\dd u)\\
    &\quad = \int_0^t\int_0^\infty\int_{|v|>1} \mathbf 1_{\{r Y_{s-}\leq 1\}}Y_{s-}(\ee^v-1)
                        \cP(\dd s,\dd r,\dd v)
   \end{split}
 \end{align}
 for all $t<S_0$.
By the definition of the integral with respect to the Poisson random measure {$\mathcal{N}_0$},
 the left-hand side of \eqref{help_large_jumps} equals
 \begin{align*}
    \sum_{\{n\in\N \;:\; \widetilde t_n \leq \tau(tz^{-1})\}}
           z\ee^{\xi_{\widetilde t_n-}}(\ee^{\widetilde \Delta_n} - 1)
            \mathbf 1_{\{ \vert\widetilde\Delta_n \vert\geq 1\}},
 \end{align*}
 where $(\widetilde t_n,\widetilde \Delta_n)$, $n\in\N$, is a labeling of the pairs associated to jump times
 and jump sizes of the L\'evy process $\xi$ such that $\widetilde t_n = \tau(t_nz^{-1})$
 and $\widetilde \Delta_n = \Delta_n$ (for the definition of $t_n$ and $\Delta_n$, see Step 3a).
This sum has only finitely many terms almost surely, since $\tau(tz^{-1})<\infty$, $t\in[0,S_0)$,
 and on each interval a L\'evy process has only finitely many jumps which are of magnitude
 {greater than or equal to 1.}
By the definition of $Y$ and $\cP$, the right-hand side of \eqref{help_large_jumps} equals
 the following sum
 \begin{align*}
   \sum_{\{n\in\N \;:\; t_n \leq t \}} z \ee^{\xi_{\tau(t_n z^{-1})}}(\ee^{\Delta_n} -1)
                      \mathbf 1_{ \{ \vert \Delta_n \vert \geq 1\} },
 \end{align*}
 where we used that $P(R_n\leq1)=1$, $n\in\N$.
The equality of the above two sums follows readily, since
 $\widetilde t_n = \tau(t_nz^{-1})$ and $\widetilde \Delta_n = \Delta_n$, $n\in\N$.

Note also that the additional Poisson random measure $\mathcal Q$ in the definition of $\cP$ is only added
 in order to get the compensator measure $\dd s\otimes \dd r \otimes \Pi(\dd u)$ (i.e., to make the rate fit);
 all  jumps according to $\mathcal Q$ do not influence the integral on the right-hand side
 of \eqref{help_large_jumps} as they are ruled out by the indicator.

\smallskip

\textbf{Step 4d (transforming the integral for small jumps):}
It remains to check the almost sure identity
 \begin{align}\label{help4}
  \begin{split}
    			&\int_0^{ \tau({tz^{-1})}}
     			\int_{|u|\leq 1} \eta_{s-}(\ee^u-1)(\mathcal{N}_0 - \mathcal{N}_0')(\dd s,\dd u)\\
      			&  = \int_0^t\int_0^\infty\int_{|v|\leq 1} \mathbf 1_{\{r Y_{s-}\leq 1\}}Y_{s-}(\ee^v-1)
                      (\cP - \cP^{'})(\dd s,\dd r,\dd v),        \quad 0\leq t < S_0.
   \end{split}
 \end{align}
The left- and right-hand sides of \eqref{help4} are square-integrable local martingales {in $t$}
 on the time interval $[0,\infty)$.
For the right-hand side this can be checked as in Step 1c, for the left-hand side one can argue
 as follows.
Namely, by Step 1a, we have $(\delta_m)_{m\in\N}$ is a sequence of stopping times such that
 $P(\lim_{m\to\infty}\delta_m=\infty)=1$ implying that $(\tau((\delta_m\wedge S_0)z^{-1}))_{m\in\N}$
 is a sequence of stopping times such that
 $P(\lim_{m\to\infty}\tau((\delta_m\wedge S_0)z^{-1})=\infty)=1$
 and hence it is enough to show that
 \begin{align*}
   E\left(\int_0^{\tau(tz^{-1})\wedge \tau((\delta_m\wedge S_0)z^{-1})}
          \int_{|u|\leq 1} \eta_s^2(\ee^u -1)^2\,\dd s\,\Pi(\dd u)
   \right)
  <\infty,\qquad m\in\N,\;t\geq 0,
 \end{align*}
 which (as in Step 1c) follows from $\int_{|u|\leq 1} (\ee^u-1)^2\,\Pi(\dd u) < \infty$
and
 \begin{align}\label{help25}
  \begin{split}
    & E\left(\int_0^{\tau(tz^{-1})\wedge \tau((\delta_m\wedge S_0)z^{-1})}
           \eta_s^2 \,\dd s\right)\\
   & \quad = E\left(\int_0^{\tau((t\wedge\delta_m\wedge S_0)z^{-1})}
           \eta_s^2 \,\dd s\right)
          =   E\left(\int_0^{t\wedge \delta_m\wedge S_0}
             Y_u \,\dd u\right)
           \leq tm,
  \end{split}
 \end{align}
 where the second equality follows by a change of variable using \eqref{help23}.
This yields that, by definition of the localized Poissonian integrals
 (see, e.g. Ikeda and Watanabe \cite[page 63]{IW}), for all $m\in\N$,
 \begin{align}\label{help_large_jumps_2}
    \int_0^{ \tau({tz^{-1})}\wedge \tau((\delta_m\wedge S_0)z^{-1})}
     			\int_{|u|\leq 1} \eta_{s-}(\ee^u-1)(\mathcal{N}_0 - \mathcal{N}_0')(\dd s,\dd u)
 \end{align}
 is the $L^2$-limit, as $\varepsilon\downarrow 0$, of
 \begin{align}\label{help_large_jumps_3}
   \int_0^{ \tau({tz^{-1})}\wedge \tau((\delta_m\wedge S_0)z^{-1})}
     			\int_{\varepsilon\leq |u|\leq 1} \eta_{s-}(\ee^u-1)(\mathcal{N}_0 - \mathcal{N}_0')(\dd s,\dd u);
  \end{align}
 and
 \begin{align}\label{help_large_jumps_4}
   \int_0^{t\wedge \delta_m}\int_0^\infty\int_{|v|\leq 1}
                      \mathbf 1_{\{r Y_{s-}\leq 1\}}Y_{s-}(\ee^v-1)
                      (\cP - \cP^{'})(\dd s,\dd r,\dd v)
 \end{align}
 is the $L^2$-limit, as $\varepsilon\downarrow 0$, of
 \begin{align}\label{help_large_jumps_5}
   \int_0^{t\wedge \delta_m}\int_0^\infty\int_{\varepsilon\leq |v|\leq 1}
         \mathbf 1_{\{r Y_{s-}\leq 1\}}Y_{s-}(\ee^v-1)
                      (\cP - \cP^{'})(\dd s,\dd r,\dd v).
 \end{align}
Note also that for all $m\in\N$ and $0<\varepsilon<1$, the integrals \eqref{help_large_jumps_3} and \eqref{help_large_jumps_5}
 can be written into two parts separating the integration with respect to the Poisson random measures
 and their intensities (see for instance Jacod and Shiryaev \cite[Proposition II.1.28]{JS}) since
 for all $t\geq 0$, $m\in\N$, and $0<\eps<1$, similarly to \eqref{help25}, we have
 \[
   E\left(\int_0^{\tau(tz^{-1})\wedge \tau((\delta_m\wedge S_0)z^{-1})}
          \int_{\varepsilon\leq |u|\leq 1} \eta_{s-}\vert\ee^u -1\vert\,\dd s\,\Pi(\dd u)
   \right)\leq t\int_{\eps\leq |u|\leq 1} |\ee^u-1|\,\Pi(\dd u)<\infty
 \]
 and
 \[
  E\left( \int_0^{t\wedge\delta_m}\int_0^\infty\int_{\varepsilon\leq |v|\leq 1}
      \mathbf 1_{\{r Y_{s-}\leq 1\}}Y_{s-}\vert\ee^v-1\vert
                         \dd s\,\dd r\,\Pi(\dd v)\right)\leq t\int_{\eps\leq |u|\leq 1} |\ee^u-1|\,\Pi(\dd u)
                 <\infty.
 \]
Hence, by the same arguments as in Step 4c and using also \eqref{help2}, for all $m\in\N$, \eqref{help_large_jumps_2}
 equals the $L^2$-limit, as $\varepsilon\downarrow 0$, of
 \begin{align*}
   &\sum_{ \{n\in\N \,:\, \widetilde t_n \leq  \tau((t\wedge \delta_m\wedge S_0) z^{-1})\}}
            \!\!\!\!\!\!\!z \ee^{\xi_{\widetilde t_n-}}(\ee^{\widetilde\Delta_n}-1)
              \mathbf 1_{\{\varepsilon \leq \vert \widetilde\Delta_n\vert\leq 1 \}}
  - \int_0^{\tau((t\wedge \delta_m\wedge S_0)z^{-1})}\!\!\!\!\! \eta_s \,\dd s
              \int_{\varepsilon \leq \vert u\vert\leq 1} \!\!\!(\ee^u -1)\Pi(\dd u) \\
    &=\sum_{\{n\in\N \,:\, \widetilde t_n \leq  \tau((t\wedge \delta_m\wedge S_0)z^{-1})\}}
             z\ee^{\xi_{\widetilde t_n-}}(\ee^{\widetilde\Delta_n}-1)
             \mathbf 1_{\{\varepsilon \leq \vert \widetilde\Delta_n\vert\leq 1 \}}
            - (t\wedge\delta_m\wedge S_0) \int_{\varepsilon \leq  \vert u\vert\leq 1} (\ee^u -1)\Pi(\dd u).
 \end{align*}
Similarly, using that
 $\cP'(\dd s,\dd r,\dd v) = \dd s\otimes\dd r\otimes\Pi(\dd v)$,
 we find that for all $m\in\N$, \eqref{help_large_jumps_4} equals the $L^2$-limit,
 as $\varepsilon\downarrow 0$, of
 \begin{align*}
   &\int_0^{t\wedge \delta_m}\int_0^\infty \int_\R
             \mathbf 1_{ \{rY_{s-}\leq 1\} }
              \mathbf 1_{ \{\varepsilon\leq \vert v\vert\leq 1\} } Y_{s-}
              (\ee^v -1) \cP(\dd s,\dd r,\dd v)\\
             &\phantom{=\;}
              - \int_0^{t\wedge\delta_m} \int_0^\infty \int_\R
              \mathbf 1_{ \{rY_{s-}\leq 1\} }
              \mathbf 1_{ \{\varepsilon\leq \vert v\vert\leq 1\} } Y_{s-}
              (\ee^v -1) \,\dd s\,\dd r\,\Pi(\dd v)\\
             & = \sum_{ \{n\in\N \,:\, t_n \leq t\wedge\delta_m \} }
                 \mathbf 1_{\{\varepsilon\leq \vert \Delta_n\vert\leq 1 \}}
                  Y_{t_n-}(\ee^{\Delta_n}-1)\\
             &\phantom{=\;}
                 + \int_0^{t\wedge\delta_m}\int_0^\infty\int_\R  \mathbf 1_{ \{rY_{s-}\leq 1\} }
                   \mathbf 1_{ \{\varepsilon\leq \vert v\vert\leq 1\} } Y_{s-}(\ee^v -1)
                   (\mathbf 1_{ \{ rY_{s-}>1 \} } +  \mathbf 1_{\{Y_{s-}=0 \} } )
                   \cQ(\dd s,\dd r,\dd v)\\
             &\phantom{=\;}
                 - (t\wedge\delta_m) \int_{\varepsilon\leq \vert v\vert\leq 1}(\ee^v -1)\Pi(\dd v)\\
             & = \sum_{ \{n\in\N \,:\, t_n \leq t\wedge\delta_m \} }
                 \mathbf 1_{\{\varepsilon\leq \vert \Delta_n\vert\leq 1 \}}
                  Y_{t_n-}(\ee^{\Delta_n}-1)
                 -  (t\wedge\delta_m) \int_{\varepsilon\leq \vert v\vert\leq 1}(\ee^v -1)\Pi(\dd v),
 \end{align*}
 where (for the sum) we used that $R_n/Y_{t_n-} \leq 1/Y_{t_n-}$, $n\in\N$, almost surely.
Using that $\widetilde t_n = \tau(t_nz^{-1})$, $n\in\N$, and $\widetilde\Delta_n=\Delta_n$, $n\in\N$,
 by the uniqueness part of the definition of a square-integrable local martingale, we have \eqref{help4}.

\smallskip

\textbf{Step 4e (putting the pieces together):}
Using Steps 4a,b,c,d and {that for all $0\leq t<S_0$,
 \[
    \int_0^t\int_0^\infty \int_{\vert v\vert>1}
             \mathbf 1_{ \{rY_{s-}\leq 1\} }  Y_{s-}
              (\ee^v -1) \cP'(\dd s,\dd r,\dd v)
       = t  \int_{\vert v\vert>1} (\ee^v -1) \Pi(\dd v)<\infty,
 \]
 which holds since $E(\ee^{\xi_1})<\infty$,}
 we have
\begin{align*}
     \begin{split}
		Y_t&=z+\left(\gamma+ \frac{\sigma^2}{2} + \int_\R(\ee^u -1 - u\mathbf 1_{ \{\vert u\vert\leq 1\} })
                      \,\Pi(\dd u)\right)t
              + \sigma\int_0^{t } \sqrt{Y_s} \dd W_s\\
           &\quad+\int_0^t\int_0^{\infty}\int_{\vert v\vert\leq 1} \mathbf 1_{\{r Y_{s-}\leq 1\}}Y_{s-}(\ee^v-1)
               (\mathcal P-\mathcal P')(\dd s,\dd r,\dd v),\qquad 0\leq t<S_0.
     \end{split}
\end{align*}
By Sato \cite[Theorem 25.17]{Sat} the drift coefficient takes the desired form $\log E(\ee^{\xi_1})$,
 and hence we have the SDE \eqref{eqn} {with $q=0$}.

\smallskip

\textbf{Step 5 (end of proof):} In Step 4 we proved that the last four integrals of (\ref{eqn:22})
 can be replaced on an event of full probability by the desired integrals.
Hence, $Y$ is a weak solution of (\ref{eqn}) {with $q=0$} up to the first hitting time {$S_0$}
 of zero.
Since the pathwise uniqueness of (\ref{eqn}) implies well-posedness of the corresponding martingale problem
 (see Step 2 in the proof of Proposition \ref{SM}), by Theorem 4.6.1 of Ethier and Kurtz \cite{EK}
 also the localized martingale problem is well-posed.
Hence, $Z^\dag$ and $Y$ (extended by $0$ after $S_0$) are equal in law.
\end{proof}

{\begin{proof}[Proof of Theorem \ref{thm:1}]
The proof is a combination of Propositions \ref{prop:uniqueness}, \ref{P:existence},
 \ref{prop:self_sim_index_1}, \ref{SM} and \ref{prop:reversed}.
\end{proof}}

\begin{proof}[Proof of Theorem \ref{thm:3}]
Under the conditions of the theorem, by \eqref{new1}
 the drift coefficient $\log E(\ee^{\frac{1}{a}\xi_1};\zeta>1)$ for the SDE (\ref{eqn})
 associated to $\frac{1}{a}\xi$ killed at rate $q$ is positive.
Part (a) follows from Theorem \ref{thm:1} and Proposition \ref{prop:reversed} combined
 with the equivalence of (\ref{b}).
For part (b), taking into account the discussion before the theorem, we only have to prove that
 the solution $Z$ constructed in Theorem \ref{thm:1} leaves zero continuously.
The uniqueness of the extension follows from Theorem 2 in Rivero \cite{R1}.
If $q=0$, solutions hit zero continuously, so that by SDE (\ref{eqn}), since the integrands of both stochastic integrals vanish if $Z_{s-}=0$, the positive constant drift pushes the solution into the interior of $\R_+$.
This implies that solutions leave zero continuously.
If $q>0$, then zero is only hit by a negative jump due to the killing integral.
Since {in this case} $Z_{s-}>0$, one might think that the integral driven by $\mathcal N$ jumps away from zero.
But since $\mathcal M$ and $\mathcal N$ are two independent Poisson random measures,
 {the Poissonian integrals corresponding to them} do not jump together almost surely.
Hence, {the integral driven by} $\mathcal N$ does not jump if {$Z_{s-}>0$ and} $Z_s=0$,
 and solutions leave zero again continuously due to the positive drift.
 \end{proof}

\begin{proof}[Proof of Theorem \ref{thm:4}]
First of all, note that if $\xi$ does not drift to $-\infty$, then the convexity of the Laplace exponent of $\xi$ trivially
 implies that  $\log E(\ee^{\frac{1}{a}\xi_1};\zeta>1)>0$ since it is null at zero. Further, due to (\ref{new1}), the same is true under the assumption of the theorem if $\xi$ drifts to $-\infty$. Hence, the drift coefficient
 $\log E(\ee^{\frac{1}{a}\xi_1};\zeta>1)$ for the SDE (\ref{eqn}) associated to $\frac{1}{a}\xi$ killed at rate $q$ is strictly positive.\\
As for the continuous case of Section \ref{sec:ex},
 we refer to a result from stochastic calculus such as Theorem IX.4.8 of
 Jacod and Shiryaev \cite{JS} applied in a rather trivial fashion, since we only change the initial conditions.
In their notation (compare also their Section III.2.c), in particular Theorem III.2.26 and Remark III.2.29),
 the sequence of martingale problems  corresponding to the jump type SDE (\ref{eqn}) has initial conditions
 $\eta^z=z$, $z\geq 0$, and time-homogeneous coefficients
	\begin{align*}
		K^z(y,A)&=\int_0^\infty\int_{\R} \mathbf 1_{A\setminus \{0\}}\big(\mathbf 1_{\{r y\leq 1\}}y
                  (\ee^u-1)\big)\,\dd r \,\Pi^a(\dd u)+q\int_0^\infty \mathbf 1_{A\setminus \{0\}}\big(-\mathbf 1_{\{r y\leq 1\}}y
                 \big)\,\dd r,
\end{align*}
$A\in \mathcal B(\R),\;y\geq 0$, and
\begin{align*}
		b^z(y)&={\log \E(\ee^{\frac{1}{a}\xi_1};\zeta>1)} - \int_{\vert v\vert >1} v \,K^z(y,\dd v),\\
		c^z(y)&=y,
	\end{align*}
	$y\geq 0$. By $\Pi^a$ we denote the L\'evy measure of the L\'evy process $\frac{1}{a}\xi$.
For completeness, we extend $c^z, b^z$ and $K^z$ by zero for $y<0$.
In what follows we check the conditions of Theorem IX.4.8 of Jacod and Shiryaev \cite{JS}.
{First, we note that for all Borel measurable and non-negative functions $f:\R\to\R_+$ with $f(0)=0$,
 we have
 \begin{align}\label{K_help}
   \int_\R f(v) K^z(y,\dd v) = \int_0^\infty\int_\R f\big(\mathbf 1_{\{ry\leq 1\}} y(\ee^u-1)\big)\dd r\,\Pi^a(\dd u)+q\int_0^\infty f\big(-\mathbf 1_{\{ry\leq 1\}} y\big)\dd r
 \end{align}
 for $y,z\geq 0$,
 in the sense that the left and right hand sides are both finite or infinite at the same time.
Indeed, if $f$ is a linear combination of indicator functions of Borel subsets of $\R$ {not containing $0$},
 then it follows by the definition of $K^z(y,\cdot)$,
 if $f$ is Borel measurable, non-negative and bounded, then it is a consequence of dominated convergence theorem,
 while for a general $f$ follows by monotone convergence.}
The coefficients do not depend on $z$, $\eta^z$ converges trivially to $\eta^0$ as $z\downarrow 0$ and,
 $b^0$ is continuous.
By Theorem \ref{thm:1} and Step 1 in the proof of Proposition \ref{SM},
 we have condition (IX.4.3) of Jacod and Shiryaev \cite{JS} is satisfied for $(b^0,c^0,K^0)$.
 Next, the modification $\tilde c^z$ is given by
 \begin{align*}
 	\tilde c^z(y)&=c^z(y)+{\int_{\R}} v^2\mathbf 1_{\{|v|\leq 1\}}K^z(y,\dd v)\\
     & =y+\int_{\R} y (\ee^u-1)^2 \mathbf 1_{\{ \vert y(\ee^u-1)\vert \leq 1\}}\,\Pi^a(\dd u)
          +qy {\mathbf 1_{\{ \vert y\vert \leq 1\}}}
 \end{align*}
 for $y\geq 0$, where the second equality follows by {\eqref{K_help} integrating out with respect to $r$.}
Hence $\tilde c^z$ is continuous due to dominated convergence and our assumption
that $\xi$ is spectrally negative, which implies $\int_{\R} (\ee^u-1)^2\,\Pi^a(\dd u)<\infty$.
Similarly, for all $z\geq 0$ and continuous bounded non-negative test-functions $g$
 vanishing around zero, by dominated convergence, the function
	\begin{align*}
		(0,\infty)\ni y\mapsto
            {\int_{\R}} g(v)K^z(y,\dd v)=\frac{1}{y}\int_{\R}  g(y (\ee^u-1))\Pi^a(\dd u)+q\frac{g(-y)}{y}
	\end{align*}
 is continuous having limit $0$ as $y\downarrow 0$.
To check condition (IX.4.9) in Jacod and Shiryaev \cite{JS}, we use
	\begin{align*}
		K^0(y,\{v:|v|>b\})=
                           \frac{1}{y}\int_{\R} \mathbf 1_{\{ \vert\ee^u-1\vert >b/y\}}\, \Pi^a(\dd u)+   q    \frac{\mathbf 1_{\{b<y\}}}{y},
                          \quad y>0,\; b>0,
	\end{align*}
 which, uniformly in ${y\in[0,N]}$ {for any $N>0$}, goes to zero as $b\to \infty$.
{Indeed, the second summand trivially tends to zero and for the first summand we obtain
\begin{align*}
  \frac{1}{y}\int_{\R} \mathbf 1_{\{ \vert\ee^u-1\vert >b/y\}}\, \Pi^a(\dd u) \leq   \frac{1}{y}\int_{\R} \frac{y^2(\ee^u-1)^2}{b^2}\, \Pi^a(\dd u)  = \frac{y}{b^2}\int_{\R} (\ee^u-1)^2\, \Pi^a(\dd u)
\end{align*}
which tends to zero as $b\to\infty$.}
\end{proof}

\noindent {\bf Acknowledgement.}
The authors would like to thank to Jean Bertoin, Zenghu Li, Pierre-Henri Cumenge and Lorenzo Zambotti
 for valuable discussions on the topic.
Special thanks to Mladen Savov for discussions on Proposition \ref{prop:uniqueness}.
We are grateful to the referee and one of the associate editors for valuable
comments that have led to an improvement of the manuscript.

\end{document}